\documentclass[a4paper,11pt]{article}
\usepackage{makeidx}
\usepackage{graphicx}
\usepackage{subcaption}
\usepackage[english,activeacute]{babel}
\usepackage[latin1]{inputenc}
\usepackage{amsmath}
\usepackage{amsfonts}
\usepackage{amssymb}
\usepackage{amsthm}
\usepackage{latexsym}
\usepackage{colortbl}
\usepackage{multicol}
\usepackage{hyperref}
\usepackage{xcolor}
\usepackage{booktabs}
\usepackage{orcidlink}
\usepackage{mathrsfs}
\usepackage{tikz}
\usepackage{pgfplots}
\pgfplotsset{compat=1.18}
\usepackage{mathtools}

\newtheorem{theorem}{Theorem}[section]
\newtheorem{proposition}[theorem]{Proposition}
\newtheorem{lemma}[theorem]{Lemma}

\newcommand{\N}{\mathbb{N}}

\newcommand{\C}{\mathbb{C}}
\newcommand{\PP}{\mathbb{P}}

\newcommand{\Charlier}[1]{C^{\mu}_{#1}}            % monic Charlier
\newcommand{\Meixner}[1]{M^{\gamma,\mu}_{#1}}      % monic Meixner
\newcommand{\MeixnerS}[2]{M^{#1,\mu}_{#2}}         % monic Meixner, shifted first parameter
\newcommand{\CharlierA}[1]{C^{\mu,A}_{#1}}         % modified Charlier (mass point at 0)
\newcommand{\MeixnerA}[1]{M^{\gamma,\mu,A}_{#1}}   % modified Meixner  (mass point at 0)
\newcommand{\Ker}{\mathrm{Ker}}                    % kernel polynomials
\newcommand{\dif}{\Delta}                          % forward difference
\newcommand{\nab}{\nabla}
\newcommand{\Nzero}{\mathbb{N}_{0}}                          % backward difference
\newcommand{\ff}[2]{\left[#1\right]_{#2}}          % falling factorial

\newcommand{\rmd}{\mathrm{d}}

\begin{document}
\pagestyle{plain}

\title{Mehler--Heine asymptotics for finite differences of
mass-modified Charlier and Meixner polynomials}
\author{{Anier Soria-Lorente %
\orcidlink{0000-0003-3488-3094}}$^{1}$, {Junior Michel \orcidlink{0009-0005-8450-2757}}$^{1}$ \\
%EndAName
\\
$^{1}$Department of Quantitative Methods, Loyola University,\\ Avda. de
las Universidades, 2. Dos Hermanas, Seville,\\ 41704, Andalusia, Spain.\\
asoria@uloyola.es, jmichel@uloyola.es
}

\maketitle
\begin{abstract}
	Point-mass perturbations and finite-difference operations alter the
	finite-degree structure of discrete orthogonal polynomials, but their combined
	effect at the Mehler--Heine scale is not immediate. We study the monic Charlier
	and Meixner families under a pure Uvarov modification by a fixed mass
	$A\delta_{0}$ at the endpoint of the support and determine the asymptotic
	behaviour of their forward and backward differences of arbitrary fixed order.
	Using a rank-one connection formula, we show that the perturbation coefficient
	decays factorially in the Charlier case and exponentially, with an algebraic
	prefactor, in the Meixner case. In both families this decay is faster than
	every algebraic power of $n^{-1}$ and therefore suppresses the growth produced
	by any fixed number of finite differences. Consequently, for every fixed
	$A\geq0$ and $k\in\Nzero$, the mass-modified and classical families have the
	same locally uniform Mehler--Heine limits in $\C$. Forward differences preserve
	the reciprocal-Gamma profile up to the factor $(-1)^k$, whereas backward
	differences translate the limiting argument by $k$, shifting the limiting
	zero lattice from $\Nzero$ to $k+\Nzero$. We also derive a
	first-order shift equation for the common forward limit and characterize its
	entire solution space. Complex-plane portraits and real-axis computations
	illustrate the predicted limiting profiles, the two zero lattices, and the
	asymptotic disappearance of the fixed endpoint mass. These results identify a
	scale-separation mechanism governing the fixed-order asymptotic stability of
	the Charlier and Meixner families under rank-one endpoint perturbations.
\end{abstract}

\vspace{0.3cm}

\par

\smallskip
{\footnotesize Corresponding author: Anier Soria-Lorente\par}
%\end{abstract}

%%%%%%%%%%%%%%%%%%%%%%%%%%%%%%%%%%%%%%%%%%%%%%%%%%%%%%%%%%%%%%%%%%%%%%%%%%%%%%%%%%%%%%%%%%%%%%%%%%%%%%%

%%%%%%%%%%%%%%%%%%%%%%%%%%%%%%%%%%%%%%%%%%%%%%%%%%%%%%%%%%%%%%%%%%%%%%%%%%%%%%%%%%%%%%%%%%%%%%%%%%%%%%%

\section{Introduction}

Orthogonal polynomials on discrete lattices combine explicit hypergeometric
representations, three-term recurrence relations, second-order difference
equations, Christoffel--Darboux kernels, and shift identities within a
remarkably rigid analytic structure \cite{Ismail05,KLS2010}. The Charlier and
Meixner families are canonical representatives of this theory. Besides their
role in the Askey scheme, their orthogonality measures are connected with the
Poisson and negative-binomial distributions, respectively, and their explicit
forward and backward shift relations make them particularly suitable for
testing the stability of classical properties under perturbation.

Among the asymptotic tools available for orthogonal polynomials,
Mehler--Heine formulas are especially sensitive to local structure. They
describe the large-degree behaviour in a fixed complex neighbourhood of an
endpoint and, through the zeros of the limiting function, retain fine
information about the underlying lattice. For the monic Charlier and Meixner
polynomials, Dominici derived reciprocal-Gamma Mehler--Heine limits and used
them to describe the corresponding local zero pattern
\cite{Dominici2015}. Higher-order terms in these expansions were subsequently
obtained in \cite{Dominici2018}. With the normalisation adopted in this paper,
the corresponding classical limiting functions are defined in
\eqref{eq:phis}; both are entire and have the nonnegative integer lattice
$\Nzero$ as their zero set.

The question considered here arises from the interaction of two operations
that are individually well understood but whose combined asymptotic effect is
not immediate. The first is a rank-one endpoint Uvarov modification of the
orthogonality measure. The second is repeated finite differencing. Although a
fixed endpoint mass produces an explicit finite-degree correction, repeated
differences lower the polynomial degree and introduce algebraic factors in
$n$; in the Meixner case they also shift a family parameter. It is therefore
not evident a priori whether the mass perturbation remains visible, is
amplified, or disappears in the Mehler--Heine regime.

\subsection{Related Work}

Point-mass modifications of classical discrete orthogonality measures have a
substantial structural history. For the Meixner, Krawtchouk, and Charlier
systems, the modification obtained by adding a Dirac mass at $x=0$ was studied
in detail by \'Alvarez-Nodarse, Garc\'ia, and Marcell\'an
\cite{alvareznodarse1995}. They derived representations of the resulting
orthogonal polynomials, hypergeometric descriptions, second-order difference
equations, recurrence relations, and the corresponding perturbation of the
associated tridiagonal matrices. In particular, their structural formulas
provide the rank-one representation used below in \eqref{eq:perturb}.
Consequently, the algebraic effect of the endpoint mass is well understood,
whereas its interaction with repeated finite differencing at the
Mehler--Heine scale requires a separate asymptotic analysis.

Finite-difference perturbations have also generated a substantial literature
within discrete Sobolev orthogonality. The systematic study of orthogonality
with respect to inner products involving differences goes back to Bavinck and
Bavinck--Haeringen \cite{B1995,B1995gen,B1996}, and the Meixner--Sobolev
instance was analysed in detail, including its ratio and Plancherel--Rotach
asymptotics, in \cite{agm,agmmb}; general surveys of Sobolev orthogonal
polynomials and their asymptotics are \cite{almarero,maxu,MR1990,mapepi}, with
early generalisations of the classical families in \cite{khol}. A
generating function for a nonstandard $\Delta$-Meixner--Sobolev family was
obtained in \cite{mbtpmp}, and a Mehler--Heine formula together with
consequences for the zeros was established for the corresponding first-order
setting in \cite{mbDMS}; analogous first-order results and further analytic
properties for Charlier--Sobolev families appear in \cite{domo,HS2019}, and
second-order difference equations for Sobolev-type families in \cite{Rebocho2022}, and
related second-order difference equations for quasi-orthogonal families
attached to the Hahn operator in \cite{gamamo}.
These results show that finite-difference modifications can preserve important
elements of the classical asymptotic structure, although the precise limiting
behaviour depends on the nature of the perturbation.

The reference most closely related to the present work is the study by
Costas-Santos, Soria-Lorente, and Vilaire \cite{Costas-2022} of monic Meixner
polynomials orthogonal with respect to the Sobolev-type inner product
\begin{equation*}
	\langle f,g\rangle
	=\langle \mathbf{u}^{M},fg\rangle
	+\lambda\,\mathscr{T}^{j}f(\alpha)\,\mathscr{T}^{j}g(\alpha),
	\quad \lambda>0,\ j\in\N,\ \alpha\le0,
\end{equation*}
where $\mathbf{u}^{M}$ is the Meixner functional and $\mathscr{T}$ is either
the forward operator $\dif$ or the backward operator $\nab$. Among many
structural results, they obtain a connection formula, a hypergeometric
representation, ladder operators, a second-order difference equation, a
$(2j+3)$-term recurrence relation, and, in the special case $\alpha=0$, a
Mehler--Heine formula whose limit is the classical reciprocal-Gamma function
\cite[Thm.~5]{Costas-2022}. That setting is, however, structurally different
from the one considered here. In \cite{Costas-2022} the higher-order
difference $\mathscr{T}^{j}$ forms part of the inner product itself, so the
perturbation acts through the associated discrete kernels
$\mathscr{K}^{(j,j)}$ and modifies the orthogonality relation. In the present
paper the modification of the measure is the pure Uvarov mass $A\delta_{0}$,
with no difference operator inside the inner product; the resulting sequence
is classical-plus-rank-one through \eqref{eq:perturb}, and the operators
$\dif^{k}$ and $\nab^{k}$ are applied \emph{afterwards}, externally, to the
already-modified polynomials. Consequently, our arbitrary-order forward and
backward Mehler--Heine formulas are not a reformulation of the single-node
Sobolev result: they describe the asymptotics of external finite differences
of Uvarov-modified polynomials and exhibit a forward--backward asymmetry
(limiting zero set $\Nzero$ versus $k+\Nzero$) that has no counterpart in the
$\alpha=0$ Sobolev Mehler--Heine formula, where the limit is the classical
profile itself.

More recently, higher-order discrete Sobolev systems have revealed that the
location and form of the perturbation can leave a visible signature in the
limiting function. In particular, for Sobolev-type Charlier and Meixner
families involving arbitrary-order forward differences evaluated at an
exterior point $\alpha<0$, the limiting function contains the factor
$z-\alpha$, producing an additional asymptotic zero at the mass point
\cite{SoriaMichel2026}. The pure endpoint Uvarov modification considered in
the present paper is structurally different from both of these Sobolev
settings: the perturbation is the fixed mass $A\delta_0$ located at the
endpoint of the support, while the forward and backward differences are
subsequently applied to the resulting modified polynomial sequence. Whereas an
exterior Sobolev mass leaves its position encoded in the limit through the
factor $z-\alpha$, and a single-node difference-Sobolev inner product keeps the
classical profile at $\alpha=0$, the endpoint Uvarov mass studied here
disappears entirely from every fixed-order limit, for each fixed $A\ge0$ and
fixed $k\in\Nzero$. This distinction provides a
natural setting in which to determine precisely which features of the
perturbation survive in the local large-degree limit.

We should be explicit about the relationship between the present work and the
two most closely related references, \cite{Costas-2022} and
\cite{SoriaMichel2026}, delimiting what is genuinely new here from what is a
direct consequence of already-known identities. The reference
\cite{SoriaMichel2026} is a recent preprint that shares authors with the present
paper; to avoid any impression of fragmentation of closely related results, we
state the distinction sharply. In \cite{SoriaMichel2026} the difference operator
sits inside a Sobolev inner product evaluated at an \emph{exterior} node
$\alpha<0$, and the mass leaves a permanent trace in the limit through the factor
$z-\alpha$; in \cite{Costas-2022} the higher-order difference is again part of
the inner product, at the single node $\alpha=0$, and the limit is the classical
profile. In the present paper there is no difference operator inside the inner
product at all: the measure is modified by the pure endpoint Uvarov mass
$A\delta_0$, and the differences $\dif^{k}$, $\nab^{k}$ are applied
\emph{externally} to the already-modified polynomials. The three settings, and
their limiting functions, are contrasted in Table~\ref{tab:comparison}.

\begin{table}[h]
	\centering
	\caption{Comparison of the present work with the two most closely related
		references. ``In/out'' indicates whether the difference operator acts
		inside the inner product or is applied externally to the modified
		polynomials. Here $\Nzero$ and $k+\Nzero$ denote the limiting zero
		lattices of the forward and backward differences, respectively.}
	\label{tab:comparison}
	\renewcommand{\arraystretch}{1.4}
	\begin{tabular}{lccc}
		\toprule
		& \cite{Costas-2022} & \cite{SoriaMichel2026} & This paper\\
		\midrule
		Perturbation type & Sobolev-type & Sobolev-type & pure Uvarov\\
		Mass location & node $\alpha=0$ & exterior $\alpha<0$ & endpoint $x=0$\\
		Operator in/out of $\langle\cdot,\cdot\rangle$ & inside & inside & outside\\
		Difference order & fixed $j$ & arbitrary fixed & arbitrary fixed $k$\\
		Trace of mass in limit & none & factor $z-\alpha$ & none\\
		Limiting function & classical profile & classical $\times(z-\alpha)$ &
		$(-1)^k\varphi(z)$, $(-1)^k\varphi(z-k)$\\
		\bottomrule
	\end{tabular}
\end{table}

Within this delimitation, the specific contribution of the present paper is
narrow but clean, and it is worth separating its conceptual part from what is an
immediate consequence of classical identities. The conceptual novelty is the
identification of the scale-separation mechanism: once the rank-one
representation \eqref{eq:perturb} is available, the mass correction is a single
scalar multiple of a backward difference, and the factorial (Charlier) or
exponential (Meixner) decay of that scalar, established in Lemma~\ref{lem:coeff},
is placed strictly below the algebraic growth generated by any fixed number of
external differences. This is what makes the endpoint mass invisible in every
fixed-order limit and is genuinely specific to the endpoint Uvarov setting, as
the contrast with \cite{SoriaMichel2026} shows. By contrast, the core of the
proofs in Sections~\ref{sec:mhfinite}--\ref{sec:mhmass} is, once
\eqref{eq:perturb} is known, an exact combination of the classical
forward/backward shift identities \eqref{eq:fwdCharlier}--\eqref{eq:bwdMeixner}
with the decay of $D_n$ and $B_n$; we do not claim novelty for that mechanical
part. In particular, the first-order functional equation of
Section~\ref{sec:functional} is \emph{not} a new asymptotic reduction of the
second-order difference equation satisfied by the finite-degree modified
polynomials: it follows directly from the explicit reciprocal-Gamma
representation of the limit, as we make clear there.

\subsection{Our Contributions}

Let $\{P_n\}_{n\geq0}$ denote either the monic Charlier or the monic Meixner
family, and let $\{P_n^A\}_{n\geq0}$ be the corresponding sequence obtained
after adding a fixed nonnegative mass $A\delta_0$ at the origin. We determine
the Mehler--Heine behaviour of the $k$th forward and backward finite
differences of $P_n^A$ for arbitrary fixed $k\in\Nzero$ as $n\to\infty$. The
analysis is carried out locally uniformly in $\C$.

The key mechanism is the one-term rank-one connection formula
\eqref{eq:perturb}, which expresses the mass-modified polynomial as the
corresponding classical polynomial plus a scalar multiple of its first
backward difference. This reduces the problem to a comparison between the
growth introduced by repeated finite differencing and the decay of the two
perturbation coefficients. Their precise asymptotic behaviour is established
in \eqref{eq:Dnasy} and \eqref{eq:Bnasy}. The Charlier coefficient has
factorial decay, whereas the Meixner coefficient has exponential decay with
an algebraic prefactor. In both cases the decay is faster than every
algebraic power of $n^{-1}$. This scale separation is the central mechanism
behind the asymptotic stability proved in the paper: it suppresses the
algebraic growth produced by any prescribed fixed number of finite
differences.

Using the reciprocal-Gamma normalisation $\theta_n$ introduced in
\eqref{eq:theta}, we prove that the mass-modified and classical families have
exactly the same Mehler--Heine limits at every fixed difference order.
More precisely, the arbitrary-order forward limits are given by
\eqref{eq:mhdiffC}--\eqref{eq:mhdiffM}, whereas the corresponding backward
limits are given by
\eqref{eq:mhnablakC}--\eqref{eq:mhnablakM}. Thus, for every fixed
$A\geq0$ and $k\in\Nzero$, the endpoint mass disappears completely from the
limiting functions even though it changes the finite-degree polynomials when
$A>0$. We stress that both $A$ and $k$ are held fixed throughout; the
double-scaling regimes $A=A_n$ or $k=k_n$, in which the scale separation may
break down, are not treated here and are left as open problems.

This invariance is accompanied by a genuine forward--backward asymmetry.
Forward differencing preserves the argument of the classical reciprocal-Gamma
profile and contributes only the factor $(-1)^k$. Backward differencing,
however, translates the argument by $k$. Consequently, the limiting zero
lattice moves from $\Nzero$ to $k+\Nzero$. The shifted reciprocal-Gamma
function simultaneously provides the analytic continuation through the
apparent singularities arising in the quotient representation associated
with the falling factorial. Hence the mass perturbation becomes
asymptotically invisible, whereas the direction of the finite-difference
operator remains detectable in both the limiting function and its zero
geometry.

We also record, as a complementary observation rather than a further main
result, a first-order functional equation satisfied by the common forward
Mehler--Heine limit. We show that the limit satisfies the shift equations
\eqref{eq:reducedC} and \eqref{eq:reducedM} for the Charlier and Meixner
families, respectively, and characterize their entire solution spaces. In each
case the general entire solution consists of the corresponding reciprocal-Gamma
factor multiplied by an arbitrary one-periodic entire function, and the
Mehler--Heine normalisation selects the particular constant periodic factor
associated with the classical family. We stress that this equation is a direct
consequence of the explicit reciprocal-Gamma form of the limit and of the Gamma
functional equation; it is neither a new asymptotic reduction of the
second-order difference equation of the finite-degree polynomials nor dependent
on the Uvarov structure, since the limiting function is the classical one.

The resulting picture identifies a precise stability mechanism for
rank-one endpoint perturbations. The finite-degree mass correction is not
simply negligible: its factorial or exponential decay places it on a scale
strictly below the algebraic growth generated by every fixed-order
difference operator. This explains why the mass disappears from all the
fixed-order limits considered here. At the same time, comparison with the
exterior-point Sobolev setting of \cite{SoriaMichel2026} shows that this
invisibility is not a generic consequence of discrete perturbation: an
exterior Sobolev mass leaves its position encoded in the limiting function,
whereas the endpoint Uvarov mass studied here does not.

The remainder of the paper is organised as follows. We first collect the
classical Charlier and Meixner data, the reciprocal-Gamma normalisation, and
the endpoint Uvarov connection formula. We then establish the finite-difference
Mehler--Heine formulas for the classical families and transfer them to the
mass-modified sequences through the decay of the perturbation coefficients.
The arbitrary-order forward and backward limits are derived next, followed by
the first-order functional equation and the characterization of its entire
solutions. The final part of the paper presents the numerical illustrations
and summarizes the main conclusions.

%%%%%%%%%%%%%%%%%%%%%%%%%%%%%%%%%%%%%%%%%%%%%%%%%%%%%%%%%%%%%%%%%%%%%%%%%%%%%%%%%%%
%=======================================================================
\section{Preliminaries for Charlier and Meixner polynomials}
\label{sec:prelim}

\subsection{Classical Charlier and Meixner polynomials}
We use the monic normalisation of both families, with the data collected
in Table~\ref{tab:families}; these are the standard parameters of the
classical discrete polynomials \cite{KLS2010,nikiforov1991classical}. For Charlier the
parameter is $\mu>0$; for Meixner the parameters satisfy $\gamma>0$ and
$0<\mu<1$. The classical facts we use are collected in the next proposition; all are
standard and are recorded only to fix normalisations
\cite{Ismail05,KLS2010,nikiforov1991classical}.

\begin{table}[h]
	\centering
	\caption{Data of the monic Charlier $\Charlier n(x)$ and Meixner
		$\Meixner n(x)$ polynomials used in this paper.  Here $\alpha_{n}$,
		$\beta_{n}$ are the recurrence coefficients, $\sigma$, $\tau$,
		$\lambda_{n}$ the coefficients of the difference equation, and
		$d_{n}^{2}$ the squared norm.}
	\label{tab:families}
	\renewcommand{\arraystretch}{1.9}
	\begin{tabular}{ccc}
		\toprule
		$P_{n}(x)$ & $\Charlier n(x)$, $\mu>0$ &
		$\Meixner n(x)$, $\gamma>0$, $0<\mu<1$\\
		\midrule
		$\alpha_{n}$ & $n+\mu$ & $\dfrac{n(1+\mu)+\mu\gamma}{1-\mu}$\\
		$\beta_{n}$ & $n\mu$ & $\dfrac{n\mu(n-1+\gamma)}{(\mu-1)^{2}}$\\
		$\sigma(x)$ & $x$ & $x$\\
		$\tau(x)$ & $\mu-x$ & $(\mu-1)x+\mu\gamma$\\
		$\lambda_{n}$ & $n$ & $(1-\mu)n$\\
		$d_{n}^{2}$ & $n!\,\mu^{n}$ &
		$\dfrac{n!\,(\gamma)_{n}\,\mu^{n}}{(1-\mu)^{\gamma+2n}}$\\
		\bottomrule
	\end{tabular}
\end{table}

Throughout this work, $\dif$ and $\nab$ denote the forward and backward
difference operators, respectively, acting on a function $f$ according to
\begin{equation*}\label{eq:diffops}
	\dif f(x)=f(x+1)-f(x),
	\quad
	\nab f(x)=f(x)-f(x-1).
\end{equation*}
The two operators are related through the elementary identities
\begin{equation*}\label{eq:diffrel}
	\nab f(x)=\dif f(x-1),
	\quad
	\dif f(x)=\nab f(x+1),
\end{equation*}
so that every statement concerning one operator admits a counterpart for the
other after a suitable shift of the argument.

Higher-order differences are defined recursively by
\begin{equation*}\label{forWop}
	\dif^{k}f(x)=\dif\![\dif^{k-1}f(x)],
	\quad
	\nab^{k}f(x)=\nab\![\nab^{k-1}f(x)],
	\quad k\geq 1,
\end{equation*}
with the convention $\dif^{0}=\nab^{0}=\mathrm{Id}$. Both operators are linear,
that is, for $\alpha,\beta\in\C$,
\begin{equation}\label{eq:linearity}
	\dif\!\left[\alpha f(x)+\beta g(x)\right]
	=\alpha\,\dif f(x)+\beta\,\dif g(x),
	\quad
	\nab\!\left[\alpha f(x)+\beta g(x)\right]
	=\alpha\,\nab f(x)+\beta\,\nab g(x).
\end{equation}
On the other hand, for $k\in\mathbb{N}_0$, the symbol $[\cdot]_k$ stands for the falling factorial
of order $k$,
\begin{equation}\label{eq:falling}
	[x]_k=
	\begin{cases}
		1, & k=0,\\[2mm]
		\displaystyle\prod_{j=0}^{k-1}(x-j), & k\geq 1,
	\end{cases}
\end{equation}
which plays the role of the monomial $x^k$ in the discrete setting, since the
difference operators act on it as
\begin{equation*}\label{eq:diffpower}
	\dif\,[x]_k=k\,[x]_{k-1},
	\quad
	\nab\,[x]_k=k\,[x-1]_{k-1},
	\quad k\geq 1.
\end{equation*}
The rising factorial, or Pochhammer symbol, of order $k$ is denoted by $(x)_k$
and is related to the falling factorial by
\begin{equation*}\label{eq:pochhammer}
	(x)_k=(-1)^k\,[-x]_k
	=\frac{\Gamma(x+k)}{\Gamma(x)},
\end{equation*}
where $\Gamma$ is the Gamma function. Finally, $\PP$ denotes the linear space of
polynomials with real coefficients.

%============================================================
%  Proposition 1: classical structural properties
%============================================================
\begin{proposition}\label{prop:classical}
	Let $\{P_{n}\}_{n\ge0}$ be the monic Charlier or Meixner family
	determined by Table~\ref{tab:families}. Then:
	\begin{enumerate}
		\item \emph{Three-term recurrence.} With $P_{-1}\equiv0$, $P_{0}\equiv1$,
		\begin{equation*}\label{eq:ttrr}
			xP_{n}(x)=P_{n+1}(x)+\alpha_{n}P_{n}(x)+\beta_{n}P_{n-1}(x),\quad n\ge0.
		\end{equation*}
		
		\item \emph{Second-order difference equation of hypergeometric type.}
		\begin{equation}\label{eq:sode}
			\sigma(x)\,\dif\nab P_{n}(x)+\tau(x)\,\dif P_{n}(x)
			+\lambda_{n}P_{n}(x)=0,\quad n\ge0.
		\end{equation}
		
		\item \emph{Values at the origin.}
		\begin{equation*}\label{eq:origin}
			\Charlier n(0)=(-\mu)^{n},
			\quad
			\Meixner n(0)=\frac{\mu^{n}}{(\mu-1)^{n}}\,
			\frac{\Gamma(n+\gamma)}{\Gamma(\gamma)}.
		\end{equation*}
		
		\item \emph{Forward-shift identities.} For $0\le k\le n$,
		\begin{eqnarray}
			\dif^{k}\Charlier n(x)&=&\ff{n}{k}\,\Charlier{n-k}(x),
			\label{eq:fwdCharlier}\\
			\dif^{k}\Meixner n(x)&=&\ff{n}{k}\,\MeixnerS{\gamma+k}{n-k}(x),
			\label{eq:fwdMeixner}
		\end{eqnarray}
		where $\MeixnerS{\gamma+k}{n-k}$ is the monic Meixner polynomial of
		degree $n-k$ with first parameter shifted to $\gamma+k$.
		
		\item \emph{Backward-shift identities.} For $0\le k\le n$,
		\begin{eqnarray}
			\nab^{k}\Charlier n(x)&=&\ff{n}{k}\,\Charlier{n-k}(x-k),
			\label{eq:bwdCharlier}\\
			\nab^{k}\Meixner n(x)&=&\ff{n}{k}\,\MeixnerS{\gamma+k}{n-k}(x-k),
			\label{eq:bwdMeixner}
		\end{eqnarray}
		where $\MeixnerS{\gamma+k}{n-k}$ is the monic Meixner polynomial of
		degree $n-k$ with first parameter shifted to $\gamma+k$.
	\end{enumerate}
\end{proposition}

%============================================================
%  Proposition 2: kernel polynomials
%============================================================
\begin{proposition}\label{prop:kernels}
	Let $\{P_{n}\}_{n\ge0}$ be the monic Charlier or Meixner family of
	Table~\ref{tab:families}, and let
	\begin{equation}\label{eq:kerdef}
		\Ker_{n-1}(x,y)=\sum_{m=0}^{n-1}\frac{P_{m}(x)P_{m}(y)}{d_{m}^{2}},
	\end{equation}
	denote the associated kernel polynomials. Then:
	\begin{enumerate}
		\item \emph{Closed forms at $y=0$.} Evaluating the Christoffel--Darboux
		kernel at $y=0$ and using the origin and backward-shift identities of
		Proposition~\ref{prop:classical} gives the closed forms
		\begin{align}
			\Ker^{C}_{n-1}(x,0)
			&=\frac{(-1)^{n-1}}{n!}\,\nab\Charlier n(x),
			\label{eq:kerC}\\
			\Ker^{M}_{n-1}(x,0)
			&=\frac{(-1)^{n-1}(1-\mu)^{n+\gamma-1}}{n!}\,\nab\Meixner n(x).
			\label{eq:kerM}
		\end{align}
		
		\item \emph{Numerical values at $x=0$.} Evaluated at $x=0$, the kernels take
		the values
		\begin{align}
			\Ker^{C}_{n-1}(0,0)&=\sum_{m=0}^{n-1}\frac{\mu^{m}}{m!},
			\label{eq:kerC00}\\
			\Ker^{M}_{n-1}(0,0)
			&=(1-\mu)^{\gamma}\sum_{m=0}^{n-1}\frac{(\gamma)_{m}\,\mu^{m}}{m!}.
			\label{eq:kerM00}
		\end{align}
		
		\item \emph{Limiting values.} Since
		\begin{equation*}
			\sum_{m\ge0}\frac{\mu^{m}}{m!}=e^{\mu}\quad\mbox{and}\quad\sum_{m\ge0}(\gamma)_{m}\frac{\mu^{m}}{m!}=(1-\mu)^{-\gamma},
		\end{equation*}
		the limiting kernel
		values are
		\begin{equation*}\label{eq:kerlimit}
			\Ker^{C}_{\infty}(0,0)=\lim_{n\to\infty}\Ker^{C}_{n-1}(0,0)=e^{\mu},
			\quad
			\Ker^{M}_{\infty}(0,0)=\lim_{n\to\infty}\Ker^{M}_{n-1}(0,0)=1.
		\end{equation*}
	\end{enumerate}
	These two limits play a decisive role below: they are the exact constants
	that survive in the amplitude of the endpoint mass correction.
\end{proposition}

%============================================================
%  Mehler--Heine setup: notation and classical asymptotics
%============================================================
We use the reciprocal-Gamma normalisation of \cite{Dominici2015,mbDMS}. For
$z\in\C$ set
\begin{equation}\label{eq:theta}
	\theta_{n}(z)=\frac{\kappa^{n}}{\Gamma(n-z)},
	\quad
	\kappa=
	\begin{cases}
		-1,&\text{Charlier},\\
		\mu-1,&\text{Meixner},
	\end{cases}
\end{equation}
and write $\varphi_{C}$, $\varphi_{M}$ for the classical limit functions
\begin{equation}\label{eq:phis}
	\varphi_{C}(z)=\frac{e^{\mu}}{\Gamma(-z)},
	\quad
	\varphi_{M}(z)=\frac{1}{(1-\mu)^{\gamma+z}\,\Gamma(-z)}.
\end{equation}
Both $\varphi_{C}$ and $\varphi_{M}$ are entire, being nonvanishing multiples
of $1/\Gamma(-z)$, and their zero set is exactly $\N_{0}$; likewise
$\theta_{n}$ is entire in $z$, with zeros at $z=n,n+1,\dots$ and no poles.

\begin{proposition}[Classical Mehler--Heine type formulas]\label{prop:MH}
	With the notation of \eqref{eq:theta}--\eqref{eq:phis}, locally uniformly in
	$\C$ \cite{Dominici2015,mbDMS},
	\begin{equation}\label{eq:classicalMH}
		\lim_{n\to\infty}\theta_{n}(z)\,\Charlier n(z)=\varphi_{C}(z),
		\quad
		\lim_{n\to\infty}\theta_{n}(z)\,\Meixner n(z)=\varphi_{M}(z).
	\end{equation}
\end{proposition}

The following two lemmas collect the elementary asymptotic tools used
repeatedly in the sequel.

\begin{lemma}[Gamma-ratio expansion]\label{lem:gammaratio}
	For fixed $a,b\in\C$, as $n\to\infty$,
	\begin{equation}\label{eq:gammaratio}
		\frac{\Gamma(n+a)}{\Gamma(n+b)}
		=n^{a-b}\left(1+\frac{(a-b)(a+b-1)}{2n}+O(n^{-2})\right),
	\end{equation}
	uniformly for $a,b$ in compact sets; this is the standard Gamma-ratio expansion obtained from Stirling formula \cite{KLS2010}.
\end{lemma}

\begin{lemma}[Shift ratios]\label{lem:shiftratio}
	With $\theta_{n}$ as in \eqref{eq:theta}, the elementary shift ratios are
	\begin{equation*}\label{eq:shiftratio}
		\frac{\theta_{n}(z)}{\theta_{n}(z-1)}
		=\frac{\Gamma(n-z+1)}{\Gamma(n-z)}=n-z,
		\quad
		\frac{\theta_{n}(z)}{\theta_{n}(z+1)}
		=\frac{1}{n-z-1},
	\end{equation*}
	which follow from $\Gamma(w+1)=w\Gamma(w)$ and are read as identities between
	meromorphic functions of $z$. The first ratio, the only one used below, equals
	the polynomial $n-z$ and is entire; the second is meromorphic, with a simple
	pole at $z=n-1$, valid wherever the quotient is defined.
\end{lemma}

We shall also use the two one-step recurrences for the normalising sequence,
\begin{equation}\label{eq:thetarec}
	\theta_{n}(z)=\frac{\kappa}{\,n-1-z\,}\,\theta_{n-1}(z),
	\quad
	\theta_{n}(z)=\kappa\,\theta_{n-1}(z-1),
\end{equation}
both of which are immediate consequences of $\Gamma(w+1)=w\,\Gamma(w)$ applied
to \eqref{eq:theta}; the first keeps the argument fixed, the second shifts it by
one unit.

The second-order difference equation \eqref{eq:sode} of
Proposition~\ref{prop:classical} is of hypergeometric type in the classical
sense, with $\sigma$ of degree at most two, $\tau$ of degree one, and
$\lambda_{n}$ constant in $x$ \cite[Ch.~2]{nikiforov1991classical}.  We
emphasise that the \emph{modified} families $\CharlierA n$, $\MeixnerA n$
also satisfy second-order linear difference equations
\cite[Theorem~2]{alvareznodarse1995}, but with coefficients that fall outside this
classical hypergeometric-type class; we shall therefore refer to them
simply as second-order linear difference equations.  The Charlier
difference preserves the parameter $\mu$, whereas the Meixner difference
shifts the first parameter, $\gamma\mapsto\gamma+k$, as in
\cite[Eqs.~(4.52),(4.54)]{ANbook}.

\subsection{The endpoint Uvarov modification}
\label{subsec:masspoint}

This subsection describes how a single Uvarov mass placed at the endpoint
$x=0$ modifies the classical Charlier and Meixner families, and isolates the two
scalar constants that govern the entire asymptotic analysis that follows.
Throughout the paper we call the perturbation the \emph{endpoint Uvarov
modification} (equivalently, the addition of the \emph{endpoint mass}
$A\delta_0$); we reserve ``Sobolev-type'' for perturbations in which a
difference operator acts inside the inner product, as in
\cite{Costas-2022,SoriaMichel2026}.

Following \cite{alvareznodarse1995}, for $A\geq 0$ we consider the endpoint Uvarov modification of the classical inner product
\begin{equation}\label{eq:functional}
	\langle p,q\rangle_{A}
	=\langle p,q\rangle+A\,p(0)q(0),
	\quad p,q\in\PP.
\end{equation}
Let $\{P_n^A\}_{n\ge0}$ denote the corresponding sequence of monic orthogonal
polynomials. Expanding $P_n^A-P_n$ in the orthogonal basis
$\{P_0,\ldots,P_{n-1}\}$ and using \eqref{eq:functional} gives the rank-one
connection formula, see \cite{alvareznodarse1995},
\begin{equation*}\label{eq:fourier}
	P^{A}_{n}(x)=P_{n}(x)-A\,P^{A}_{n}(0)\sum_{k=0}^{n-1}
	\frac{P_{k}(0)P_{k}(x)}{d_{k}^{2}},
	\quad
	P^{A}_{n}(0)=\frac{P_{n}(0)}{1+A\,\Ker_{n-1}(0,0)}.
\end{equation*}
Using the kernel identities \eqref{eq:kerC}--\eqref{eq:kerM}, this expansion
reduces to the one-term backward-difference form
\cite[Eqs.~(27), (29)]{alvareznodarse1995}
\begin{equation}\label{eq:perturb}
	\CharlierA n(x)=\Charlier n(x)+D_{n}\,\nab\Charlier n(x),
	\quad
	\MeixnerA n(x)=\Meixner n(x)+B_{n}\,\nab\Meixner n(x),
\end{equation}
where the perturbation constants are \cite[Sec.~3]{alvareznodarse1995}
\begin{align}
	D_{n}&=A\,\frac{\mu^{n}}{n!\,\bigl(1+A\,\Ker^{C}_{n-1}(0,0)\bigr)},
	\label{eq:Dn}\\
	B_{n}&=A\,\frac{\mu^{n}(1-\mu)^{\gamma-1}(\gamma)_{n}}
	{n!\,\bigl(1+A\,\Ker^{M}_{n-1}(0,0)\bigr)}.
	\label{eq:Bn}
\end{align}
For $A=0$ one has $D_{n}=B_{n}=0$ and the classical polynomials are
recovered. Formula \eqref{eq:perturb} is the representation used throughout:
it isolates the perturbation into a single scalar multiple of a backward
difference of the classical polynomial, so the whole asymptotic problem
reduces to the interplay between the size of $D_{n}$, $B_{n}$ and the local
behaviour of $\nab P_{n}$.

The analysis therefore rests on the size of the constants $D_{n}$ and
$B_{n}$, whose decay mechanisms are different. The Charlier constant contains
the genuine factorial factor $\mu^{n}/n!$ and is super-exponentially small.
For Meixner, the quotient $(\gamma)_{n}/n!$ is asymptotic to a polynomial
factor, so the decay is driven by $\mu^{n}$ with $0<\mu<1$ and is merely
exponential. Both nevertheless decay faster than every power of $n^{-1}$. The
next lemma records the precise leading behaviour.

\begin{lemma}\label{lem:coeff}
	Let $A>0$ be fixed. As $n\to\infty$,
	\begin{align}
		D_{n}&=\frac{A}{1+A\,e^{\mu}}\,\frac{\mu^{n}}{n!}\,
		\bigl(1+O(\mu^{n}/n!)\bigr),
		\label{eq:Dnasy}\\
		B_{n}&=\frac{A\,(1-\mu)^{\gamma-1}}{(1+A)\,\Gamma(\gamma)}\,
		\mu^{n}\,n^{\gamma-1}\,
		\bigl(1+O(n^{-1})\bigr).
		\label{eq:Bnasy}
	\end{align}
	In particular $D_{n}\to0$ and $B_{n}\to0$ faster than any power of
	$n^{-1}$, and $nD_{n}\to0$, $nB_{n}\to0$.
\end{lemma}

\begin{proof}
	In the Charlier case, Proposition \ref{prop:kernels} shows that
	\begin{equation*}
		\Ker^{C}_{n-1}(0,0)\to e^{\mu},
	\end{equation*}
	with an error given by the Poisson tail
	\begin{equation*}
		\sum_{m\ge n}\frac{\mu^{m}}{m!}
		=
		O\!\left(\frac{\mu^{n}}{n!}\right).
	\end{equation*}
	Substituting this estimate into \eqref{eq:Dn} then yields
	\eqref{eq:Dnasy}.
	
	For the Meixner case, Proposition \ref{prop:kernels} shows that the
	denominator in \eqref{eq:Bn} converges to $1+A$. Moreover, by the
	Gamma-ratio expansion \eqref{eq:gammaratio},
	\begin{equation}\label{eq:pochleading}
		\frac{(\gamma)_{n}}{n!}
		=
		\frac{n^{\gamma-1}}{\Gamma(\gamma)}
		\left(1+O\!\left(n^{-1}\right)\right),
	\end{equation}
	which is the leading factor in \eqref{eq:Bnasy}. To control the tail of the
	kernel sum \eqref{eq:kerM00}, observe first that \eqref{eq:pochleading},
	applied with $n$ replaced by $m$, obtains that
	\begin{equation*}
		\frac{(\gamma)_{m}}{m!}
		\leq c\,m^{\gamma-1},
		\quad m\geq 1.
	\end{equation*}
	Hence, for $0<\mu<1$,
	\begin{equation*}
		\sum_{m\ge n}\frac{(\gamma)_{m}\,\mu^{m}}{m!}
		\le c\sum_{m\ge n}m^{\gamma-1}\mu^{m}
		=O\!\left(n^{\gamma-1}\mu^{n}\right),
	\end{equation*}
	the last bound being the standard estimate for the tail of a geometric
	series weighted by a fixed power (comparison with
	$\int_{n}^{\infty}t^{\gamma-1}\mu^{t}\,\rmd t$). Thus replacing the
	denominator by $1+A$ changes only an exponentially small relative term, and
	combining these facts proves \eqref{eq:Bnasy}. For Charlier, $n!$ dominates
	every fixed exponential $\mu^{n}$; for Meixner, $0<\mu<1$ and the factor
	$\mu^{n}$ dominates the polynomial factor $n^{\gamma-1}$. Hence both $D_{n}$
	and $B_{n}$ decay faster than every power of $n^{-1}$, but only the Charlier
	decay is super-exponential.
\end{proof}

\section{Mehler--Heine asymptotics for finite differences of Charlier and Meixner polynomials}
\label{sec:mhfinite}

The classical Mehler--Heine formulas \eqref{eq:classicalMH} describe the scaled
limit of the polynomials themselves. Since the endpoint Uvarov perturbation
\eqref{eq:perturb} is a scalar multiple of a backward difference, and since we
shall later apply external differences of arbitrary order to the modified
polynomials, we first determine the Mehler--Heine behaviour of the forward and
backward differences $\dif^{k}P_{n}$ and $\nab^{k}P_{n}$ of the classical
families for every fixed $k\in\Nzero$. The two operators behave differently
under the reciprocal-Gamma normalisation: each difference lowers the degree by
one and produces a factor $n$; the forward operator leaves the normalising
sequence $\theta_{n}$ and the argument $z$ unchanged, whereas the backward
operator requires division by $n^{k}$ and shifts the argument by $k$ units. We
prove the general order-$k$ statement directly; the cases $k=0$ and $k=1$, used
repeatedly below, are recorded afterwards as immediate specializations.

The four order-$k$ limits share a single structure once the family-dependent
data are collected into the normalising base $\kappa$ of \eqref{eq:theta} and
the classical limit function $\varphi$ of \eqref{eq:phis}. Let
$\{P_{n}\}_{n\ge0}$ denote either the monic Charlier family $\{\Charlier n\}$ or
the monic Meixner family $\{\Meixner n\}$ of Table~\ref{tab:families}, and write
\begin{equation}\label{eq:unifieddata}
	\varphi=
	\begin{cases}
		\varphi_{C}, & \text{Charlier},\\[1mm]
		\varphi_{M}, & \text{Meixner},
	\end{cases}
	\quad
	\kappa=
	\begin{cases}
		-1, & \text{Charlier},\\[1mm]
		\mu-1, & \text{Meixner},
	\end{cases}
	\quad
	\rho=|\kappa|=
	\begin{cases}
		1, & \text{Charlier},\\[1mm]
		1-\mu, & \text{Meixner},
	\end{cases}
\end{equation}
with $\theta_{n}$ as in \eqref{eq:theta} for the corresponding value of
$\kappa$. The single scalar $\rho$ records the only point at which the two
families differ in the backward limit: it is $1$ for Charlier and $1-\mu$ for
Meixner.

\begin{theorem}[Higher-order Mehler--Heine type formulas]\label{thm:higherMH}
	Let $k\in\N_{0}$ be fixed and let $\{P_{n}\}_{n\ge0}$ be the monic Charlier
	or Meixner family, with $\varphi$, $\kappa$ and $\rho$ as in
	\eqref{eq:unifieddata}. Then, locally uniformly in $\C$, the forward
	differences satisfy
	\begin{equation}\label{eq:unifiedFwd}
		\lim_{n\to\infty}\theta_n(z)\,\dif^kP_n(z)
		=(-1)^k\varphi(z),
	\end{equation}
	while the backward differences satisfy
	\begin{equation}\label{eq:unifiedBwd}
		\lim_{n\to\infty}\frac{\theta_n(z)}{n^k}\,\nab^kP_n(z)
		=(-1)^k\varphi(z-k).
	\end{equation}
	Written out for the two families, \eqref{eq:unifiedFwd} reads
	\begin{align}
		\lim_{n\to\infty}\theta_{n}(z)\,\dif^{k}\Charlier n(z)
		&=(-1)^{k}\,\varphi_{C}(z),
		\label{eq:higherMHCf}\\
		\lim_{n\to\infty}\theta_{n}(z)\,\dif^{k}\Meixner n(z)
		&=(-1)^{k}\,\varphi_{M}(z),
		\label{eq:higherMHMf}
	\end{align}
	and \eqref{eq:unifiedBwd} reads
	\begin{align}
		\lim_{n\to\infty}\frac{\theta_{n}(z)}{n^{k}}\,\nab^{k}\Charlier n(z)
		&=(-1)^{k}\,\varphi_{C}(z-k),
		\label{eq:higherMHC}\\
		\lim_{n\to\infty}\frac{\theta_{n}(z)}{n^{k}}\,\nab^{k}\Meixner n(z)
		&=(-1)^{k}\,\varphi_{M}(z-k).
		\label{eq:higherMHM}
	\end{align}
	For $k\ge1$ and $z\notin\{0,1,\dots,k-1\}$, the backward limit admits the
	quotient representation
	\begin{equation}\label{eq:bwdquotient}
		(-1)^k\varphi(z-k)
		=\frac{\rho^k}{[z]_k}\,\varphi(z),
	\end{equation}
	which for the two families reads
	\begin{equation*}
		(-1)^k\varphi_C(z-k)=\frac{\varphi_C(z)}{[z]_k},
		\quad
		(-1)^k\varphi_M(z-k)
		=\frac{(1-\mu)^k}{[z]_k}\,\varphi_M(z).
	\end{equation*}
	The forward limit is entire and coincides with the classical limit function
	up to the factor $(-1)^k$. The apparent singularities of \eqref{eq:bwdquotient}
	at $z=0,1,\dots,k-1$ are removable, being cancelled by the corresponding
	zeros of $1/\Gamma(-z)$, so the backward limit is entire as well. For $k=0$
	the formulas reduce to the classical limits \eqref{eq:classicalMH}.
\end{theorem}

\begin{proof}
	Iterating the relations in \eqref{eq:thetarec}, for every fixed
	$k\in\N_{0}$ we obtain
	\begin{equation}\label{eq:thetareck}
		\theta_n(z)
		=
		\frac{\kappa^k}
		{(n-1-z)(n-2-z)\cdots(n-k-z)}
		\,\theta_{n-k}(z),
		\quad
		\theta_n(z)
		=
		\kappa^k\theta_{n-k}(z-k),
	\end{equation}
	both of which follow from applying $\Gamma(w+1)=w\,\Gamma(w)$ repeatedly to
	\eqref{eq:theta}; the first keeps the argument fixed, the second shifts it by
	$k$ units. Since $k$ is fixed,
	\begin{equation}\label{eq:prodratio}
		\frac{[n]_k}
		{(n-1-z)(n-2-z)\cdots(n-k-z)}
		\longrightarrow 1,
	\end{equation}
	locally uniformly in $\C$.
	
	We first consider the forward differences. For the Charlier family,
	$\kappa=-1$. Hence, by \eqref{eq:fwdCharlier} and the first identity in
	\eqref{eq:thetareck}, we have
	\begin{equation*}
		\theta_n(z)\,\dif^k\Charlier n(z)
		=
		(-1)^k
		\frac{[n]_k}
		{(n-1-z)\cdots(n-k-z)}
		\theta_{n-k}(z)\,\Charlier{n-k}(z).
	\end{equation*}
	Using \eqref{eq:prodratio} together with Proposition
	\ref{prop:MH}, we obtain
	\begin{equation*}
		\theta_n(z)\,\dif^k\Charlier n(z)
		\longrightarrow
		(-1)^k\varphi_C(z),
	\end{equation*}
	which proves \eqref{eq:higherMHCf}.
	
	Similarly, \eqref{eq:fwdMeixner} gives
	\begin{equation*}
		\theta_n(z)\,\dif^k\Meixner n(z)
		=
		(\mu-1)^k
		\frac{[n]_k}
		{(n-1-z)\cdots(n-k-z)}
		\theta_{n-k}(z)\,
		\MeixnerS{\gamma+k}{n-k}(z).
	\end{equation*}
	Applying Proposition \ref{prop:MH} to the Meixner family with
	parameter $\gamma+k$ yields
	\begin{equation*}
		\theta_{n-k}(z)\,
		\MeixnerS{\gamma+k}{n-k}(z)
		\longrightarrow
		\varphi_M^{(\gamma+k)}(z).
	\end{equation*}
	Moreover, by \eqref{eq:phis}, we obtain
	\begin{equation}\label{eq:phishiftk}
		\varphi_M^{(\gamma+k)}(z)
		=
		\frac{\varphi_M(z)}{(1-\mu)^k}.
	\end{equation}
	Hence,
	\begin{equation*}
		(\mu-1)^k\varphi_M^{(\gamma+k)}(z)
		=
		(-1)^k\varphi_M(z),
	\end{equation*}
	and therefore \eqref{eq:higherMHMf} follows.
	
	We now turn to the backward differences. For the Charlier family,
	$\kappa=-1$. Hence, from \eqref{eq:bwdCharlier} and the second identity in
	\eqref{eq:thetareck}, we deduce
	\begin{equation*}
		\frac{\theta_n(z)}{n^k}\,
		\nab^k\Charlier n(z)
		=
		(-1)^k
		\frac{[n]_k}{n^k}
		\theta_{n-k}(z-k)\,
		\Charlier{n-k}(z-k).
	\end{equation*}
	Since $[n]_k/n^k\to1$, Proposition \ref{prop:MH}, applied at the shifted
	argument $z-k$, gives
	\begin{equation*}
		\frac{\theta_n(z)}{n^k}\,
		\nab^k\Charlier n(z)
		\longrightarrow
		(-1)^k\varphi_C(z-k),
	\end{equation*}
	which proves \eqref{eq:higherMHC}.
	
	For the Meixner family, \eqref{eq:bwdMeixner} and
	\eqref{eq:thetareck} similarly yield
	\begin{equation*}
		\frac{\theta_n(z)}{n^k}\,
		\nab^k\Meixner n(z)
		=
		(\mu-1)^k
		\frac{[n]_k}{n^k}
		\theta_{n-k}(z-k)\,
		\MeixnerS{\gamma+k}{n-k}(z-k).
	\end{equation*}
	By Proposition \ref{prop:MH}, we have
	\begin{equation*}
		\theta_{n-k}(z-k)\,
		\MeixnerS{\gamma+k}{n-k}(z-k)
		\longrightarrow
		\varphi_M^{(\gamma+k)}(z-k).
	\end{equation*}
	Using \eqref{eq:phishiftk} at $z-k$, we deduce
	\begin{equation*}
		\varphi_M^{(\gamma+k)}(z-k)
		=
		\frac{\varphi_M(z-k)}{(1-\mu)^k},
	\end{equation*}
	and consequently
	\begin{equation*}
		(\mu-1)^k
		\varphi_M^{(\gamma+k)}(z-k)
		=
		(-1)^k\varphi_M(z-k).
	\end{equation*}
	This proves \eqref{eq:higherMHM}, and hence \eqref{eq:unifiedFwd} and
	\eqref{eq:unifiedBwd} in unified form.
	
	It remains to obtain the quotient representation \eqref{eq:bwdquotient}.
	Repeated application of the Gamma functional equation gives
	\begin{equation*}
		\Gamma(k-z)
		=
		(-1)^k[z]_k\,\Gamma(-z).
	\end{equation*}
	Thus, for $z\notin\{0,1,\ldots,k-1\}$, we deduce
	\begin{equation*}
		(-1)^k\varphi_C(z-k)
		=
		\frac{\varphi_C(z)}{[z]_k},
		\qquad
		(-1)^k\varphi_M(z-k)
		=
		\frac{(1-\mu)^k}{[z]_k}\,
		\varphi_M(z),
	\end{equation*}
	which is \eqref{eq:bwdquotient} written out for each family. Since
	$\varphi_C(z-k)$ and $\varphi_M(z-k)$ are entire, the apparent singularities
	at $z=0,1,\ldots,k-1$ are removable.
	
	Finally, all the above convergences are locally uniform in $\C$: Proposition
	\ref{prop:MH} provides locally uniform convergence for the polynomial
	factors, and this persists after the fixed translation $z\mapsto z-k$;
	moreover the factor in \eqref{eq:prodratio} converges locally uniformly to
	$1$, and $[n]_k/n^k\to1$. Hence all the products above converge locally
	uniformly on every compact subset of $\C$.
\end{proof}

The backward limits in \eqref{eq:higherMHC}--\eqref{eq:higherMHM} are stated
with the power normalisation $n^{k}$. They are unchanged if one replaces
$n^{k}$ by the falling factorial $\ff{n}{k}$, since $\ff{n}{k}/n^{k}\to1$ for
every fixed $k$; the two normalisations are therefore asymptotically
equivalent and select the same limiting functions. From
Section~\ref{sec:mhmass} onwards we adopt the falling-factorial normalisation
$\ff{n}{k}$, because the backward-shift identity
\eqref{eq:bwdCharlier}--\eqref{eq:bwdMeixner} produces exactly this factor and
yields the \emph{exact} cancellation
\begin{equation*}
	\frac{\theta_{n}(z)}{\ff{n}{k}}\,\nab^{k}P_{n}(z)
	=\kappa^{k}\,\theta_{n-k}(z-k)\,P_{n-k}(z-k),
\end{equation*}
with no residual ratio to estimate.

\section{Mehler--Heine formulas for finite differences of mass-modified polynomials}
\label{sec:mhmass}

We now transfer the finite-difference Mehler--Heine formulas of
Section~\ref{sec:mhfinite} to the mass-modified families. The two ingredients
are the one-term backward-difference representation \eqref{eq:perturb}, which
isolates the perturbation into a single scalar multiple of $\nab P_{n}$, and the
order-$k$ formulas of Theorem~\ref{thm:higherMH}. Because $\dif$ and $\nab$
commute and the perturbation term in \eqref{eq:perturb} is itself a scalar
multiple of $\nab P_{n}$, every external difference of the modified polynomials
splits into a classical difference plus a perturbation term that carries one
extra backward difference. The whole point is a separation of scales, which we
state precisely before proving the theorem.

It is important to distinguish two quantities. The perturbation
\emph{coefficient} $D_{n}$ or $B_{n}$ decays beyond all algebraic orders by
Lemma~\ref{lem:coeff}. The perturbation \emph{term} that actually appears after
the Mehler--Heine normalisation is not the bare coefficient but the product of
the coefficient with a normalised difference of the classical polynomial. For a
$k$th-order external difference, the perturbation term contains one further
difference of the classical polynomial, of order $k+1$; under the
reciprocal-Gamma normalisation this normalised difference is locally uniformly
bounded (it converges to a fixed reciprocal-Gamma profile), while the leftover
algebraic factor produced by that extra difference is exactly one power of $n$.
The relevant quantity is therefore $nD_{n}$ or $nB_{n}$, not merely $D_{n}$ or
$B_{n}$. By Lemma~\ref{lem:coeff},
\begin{equation}\label{eq:nDn}
	nD_{n}
	=O\!\left(\frac{n\mu^{n}}{n!}\right)
	\longrightarrow 0,
	\qquad
\end{equation}
\begin{equation}\label{eq:nBn}
	nB_{n}
	=O\!\left(\mu^{n}n^{\gamma}\right)
	\longrightarrow 0,
	\quad 0<\mu<1,
\end{equation}
so the single algebraic power created by the extra difference is dominated by
the beyond-all-orders decay of the coefficient. This is the scale separation on
which every statement below rests: it is the boundedness of the normalised
$(k+1)$th difference together with $nD_{n}\to0$, $nB_{n}\to0$ that suppresses
the perturbation, and not a vaguer assertion that a difference of order $k+1$
merely produces algebraic growth.

We prove the general order-$k$ statement directly. The cases $k=0$ (the
polynomials themselves) and $k=1$ (their first differences) are the special
cases $k=0,1$ of Theorems~\ref{thm:mh-diff} and~\ref{thm:mh-nablak} below;
they require no separate treatment.

\begin{theorem}[Mehler--Heine formulas for $\dif^{k}P^{A}_{n}$]
	\label{thm:mh-diff}
	Fix $k\in\N_{0}$ and $A\ge0$. Then, locally uniformly for $z\in\C$,
	\begin{align}
		\lim_{n\to\infty}\theta_n(z)\,\dif^k\CharlierA n(z)
		&=
		(-1)^k\varphi_C(z),
		\label{eq:mhdiffC}\\
		\lim_{n\to\infty}\theta_n(z)\,\dif^k\MeixnerA n(z)
		&=
		(-1)^k\varphi_M(z).
		\label{eq:mhdiffM}
	\end{align}
	In particular, for $k=0$ this gives
	$\theta_n(z)\CharlierA n(z)\to\varphi_C(z)$ and
	$\theta_n(z)\MeixnerA n(z)\to\varphi_M(z)$: for each fixed $A\ge0$ the
	modified polynomials have exactly the same Mehler--Heine limits as the
	classical ones. In the forward case the normalising sequence $\theta_{n}$ can
	be kept fixed for every $k$; this is not a generic feature of finite
	differences of discrete orthogonal polynomials, but a consequence of the
	reciprocal-Gamma normalisation \eqref{eq:theta} together with the exact
	forward-shift identities \eqref{eq:fwdCharlier}--\eqref{eq:fwdMeixner}, which
	lower the degree by one at each step while, in the Meixner case, shifting the
	parameter $\gamma\mapsto\gamma+k$; this shift is absorbed into the limiting
	function through \eqref{eq:phishiftk}.
\end{theorem}

\begin{proof}
	Applying $\dif^k$ to \eqref{eq:perturb} and using
	$\dif\nab=\nab\dif$, we obtain
	\begin{equation}\label{eq:diffperturbC}
		\dif^kP_n^A(z)
		=
		\dif^kP_n(z)
		+
		C_n\,\nab\dif^kP_n(z),
		\quad
		C_n=D_n\ \text{or}\ B_n.
	\end{equation}
	After multiplication by $\theta_n(z)$, the classical term converges
	locally uniformly to $(-1)^k\varphi_C(z)$ or
	$(-1)^k\varphi_M(z)$, respectively, by
	Theorem~\ref{thm:higherMH}. It therefore remains to show that the
	perturbation term converges locally uniformly to zero.
	
	For the Charlier family, the forward and backward shift identities give
	\begin{equation*}
		\nab\dif^k\Charlier n(z)
		=
		[n]_{k+1}\,\Charlier{n-k-1}(z-1).
	\end{equation*}
	Since $\kappa=-1$, the definition \eqref{eq:theta} of $\theta_n$ yields
	\begin{equation*}
		\theta_n(z)
		=
		(-1)^{k+1}
		\frac{\Gamma(n-k-z)}{\Gamma(n-z)}
		\theta_{n-k-1}(z-1).
	\end{equation*}
	Hence,
	\begin{equation*}
		\frac{\theta_n(z)}{n}\,
		\nab\dif^k\Charlier n(z)
		=
		(-1)^{k+1}R_{n,k}(z)\,
		\theta_{n-k-1}(z-1)\,
		\Charlier{n-k-1}(z-1),
	\end{equation*}
	where
	\begin{equation*}
		R_{n,k}(z)
		=
		\frac{[n]_{k+1}}{n}\,
		\frac{\Gamma(n-k-z)}{\Gamma(n-z)}.
	\end{equation*}
	For fixed $k$, the Gamma-ratio expansion \eqref{eq:gammaratio}, applied with
	$a=-k-z$ and $b=-z$, gives
	\begin{equation*}
		\frac{\Gamma(n-k-z)}{\Gamma(n-z)}
		=n^{-k}\bigl(1+O(n^{-1})\bigr),
	\end{equation*}
	uniformly for $z$ in compact sets. Combining this with $[n]_{k+1}/n^{k+1}\to1$
	for fixed $k$ yields
	\begin{equation*}
		R_{n,k}(z)
		=\frac{[n]_{k+1}}{n^{k+1}}\bigl(1+O(n^{-1})\bigr)
		\longrightarrow1,
	\end{equation*}
	locally uniformly in $\C$. Therefore, Proposition~\ref{prop:MH},
	applied at the shifted argument $z-1$, gives
	\begin{equation*}
		\frac{\theta_n(z)}{n}\,
		\nab\dif^k\Charlier n(z)
		\longrightarrow
		(-1)^{k+1}\varphi_C(z-1),
	\end{equation*}
	locally uniformly in $\C$. In particular, this normalised $(k+1)$th
	difference is locally uniformly bounded. Writing the perturbation term as
	\begin{equation*}
		D_n\theta_n(z)\,\nab\dif^k\Charlier n(z)
		=
		nD_n\,
		\frac{\theta_n(z)}{n}\,
		\nab\dif^k\Charlier n(z),
	\end{equation*}
	the bounded normalised difference is multiplied by the scalar $nD_n$, which
	tends to $0$ by \eqref{eq:nDn}. Hence the perturbation term converges locally
	uniformly to zero, and \eqref{eq:mhdiffC} follows.
	
	For the Meixner family,
	\begin{equation*}
		\nab\dif^k\Meixner n(z)
		=
		[n]_{k+1}\,
		\MeixnerS{\gamma+k+1}{n-k-1}(z-1).
	\end{equation*}
	Since $\kappa=\mu-1$, we similarly obtain
	\begin{equation*}
		\frac{\theta_n(z)}{n}\,
		\nab\dif^k\Meixner n(z)
		=
		(\mu-1)^{k+1}R_{n,k}(z)\,
		\theta_{n-k-1}(z-1)\,
		\MeixnerS{\gamma+k+1}{n-k-1}(z-1).
	\end{equation*}
	Proposition~\ref{prop:MH}, applied to the Meixner family with shifted
	parameter $\gamma+k+1$, gives
	\begin{equation*}
		\theta_{n-k-1}(z-1)\,
		\MeixnerS{\gamma+k+1}{n-k-1}(z-1)
		\longrightarrow
		\varphi_M^{(\gamma+k+1)}(z-1).
	\end{equation*}
	Using \eqref{eq:phishiftk}, we deduce
	\begin{equation*}
		\varphi_M^{(\gamma+k+1)}(z-1)
		=
		\frac{\varphi_M(z-1)}{(1-\mu)^{k+1}},
	\end{equation*}
	and consequently
	\begin{equation*}
		\frac{\theta_n(z)}{n}\,
		\nab\dif^k\Meixner n(z)
		\longrightarrow
		(-1)^{k+1}\varphi_M(z-1),
	\end{equation*}
	locally uniformly in $\C$; this normalised $(k+1)$th difference is therefore
	locally uniformly bounded. Writing the perturbation term as
	\begin{equation*}
		B_n\theta_n(z)\,\nab\dif^k\Meixner n(z)
		=
		nB_n\,
		\frac{\theta_n(z)}{n}\,
		\nab\dif^k\Meixner n(z),
	\end{equation*}
	and using $nB_n\to0$ from \eqref{eq:nBn}, this perturbation term also
	converges locally uniformly to zero. Therefore, \eqref{eq:mhdiffM} follows.
\end{proof}

\begin{theorem}[Mehler--Heine formulas for higher-order backward differences]
	\label{thm:mh-nablak}
	Fix $k\in\N_{0}$ and $A\geq0$. Then, locally uniformly for $z\in\C$,
	\begin{align}
		\lim_{n\to\infty}
		\frac{\theta_{n}(z)}{\ff{n}{k}}\,
		\nab^{k}\CharlierA n(z)
		&=
		(-1)^{k}\varphi_{C}(z-k)
		=
		(-1)^{k}\frac{e^{\mu}}{\Gamma(k-z)},
		\label{eq:mhnablakC}\\
		\lim_{n\to\infty}
		\frac{\theta_{n}(z)}{\ff{n}{k}}\,
		\nab^{k}\MeixnerA n(z)
		&=
		(-1)^{k}\varphi_{M}(z-k)
		=
		(-1)^{k}
		\frac{1}
		{(1-\mu)^{\gamma+z-k}\Gamma(k-z)}.
		\label{eq:mhnablakM}
	\end{align}
	Thus, for each fixed $A\ge0$ and fixed $k\in\Nzero$, the endpoint mass
	perturbation does not alter the classical higher-order backward
	Mehler--Heine limits. In both cases, the limiting function is the
	corresponding classical one evaluated at the shifted argument $z-k$,
	multiplied by $(-1)^k$. In the Meixner case, the parameter shift
	$\gamma\mapsto\gamma+k$ is absorbed into the limiting function. For $k=0$
	these reduce to the polynomial case of Theorem~\ref{thm:mh-diff}, and for
	$k=1$ to the corresponding first-order backward formulas.
\end{theorem}

\begin{proof}
	Applying $\nab^k$ to the perturbation representation
	\eqref{eq:perturb} and using linearity \eqref{eq:linearity} gives
	\begin{equation}\label{eq:nablak-split}
		\begin{cases}
			\nab^{k}\CharlierA n(z)
			=
			\nab^{k}\Charlier n(z)
			+
			D_{n}\,\nab^{k+1}\Charlier n(z),\\\\
			\nab^{k}\MeixnerA n(z)
			=
			\nab^{k}\Meixner n(z)
			+
			B_{n}\,\nab^{k+1}\Meixner n(z).
		\end{cases}
	\end{equation}	
	We first consider the Charlier family. Since $\kappa=-1$, the
	higher-order backward-shift identity \eqref{eq:bwdCharlier} and the
	second relation in \eqref{eq:thetareck} yield
	\begin{equation*}
		\frac{\theta_n(z)}{\ff{n}{k}}\,
		\nab^k\Charlier n(z)
		=
		(-1)^k
		\theta_{n-k}(z-k)\,
		\Charlier{n-k}(z-k).
	\end{equation*}
	Proposition~\ref{prop:MH}, applied at the shifted argument $z-k$,
	then gives
	\begin{equation*}
		\frac{\theta_n(z)}{\ff{n}{k}}\,
		\nab^k\Charlier n(z)
		\longrightarrow
		(-1)^k\varphi_C(z-k),
	\end{equation*}
	locally uniformly in $\C$.
	
	For the perturbation term, the backward-shift identity of order
	$k+1$ gives
	\begin{align*}
		\frac{D_n\theta_n(z)}{\ff{n}{k}}\,
		\nab^{k+1}\Charlier n(z)
		&=
		(n-k)D_n\,(-1)^{k+1}
		\theta_{n-k-1}(z-k-1)
		\Charlier{n-k-1}(z-k-1).
	\end{align*}
	The normalised polynomial factor converges locally uniformly to
	$\varphi_C(z-k-1)$ by Proposition~\ref{prop:MH}, and is therefore locally
	uniformly bounded; the scalar prefactor is $(n-k)D_n$, which tends to $0$
	since $(n-k)D_n\le nD_n\to0$ by \eqref{eq:nDn}. Hence the perturbation term
	converges locally uniformly to zero. Combining both terms in
	\eqref{eq:nablak-split}, we obtain
	\begin{equation*}
		\frac{\theta_n(z)}{\ff{n}{k}}\,
		\nab^k\CharlierA n(z)
		\longrightarrow
		(-1)^k\varphi_C(z-k)
		=
		(-1)^k\frac{e^\mu}{\Gamma(k-z)},
	\end{equation*}
	which proves \eqref{eq:mhnablakC}.
	
	We now consider the Meixner family. Since $\kappa=\mu-1$,
	\eqref{eq:bwdMeixner} and the second relation in
	\eqref{eq:thetareck} give
	\begin{equation*}
		\frac{\theta_n(z)}{\ff{n}{k}}\,
		\nab^k\Meixner n(z)
		=
		(\mu-1)^k
		\theta_{n-k}(z-k)\,
		\MeixnerS{\gamma+k}{n-k}(z-k).
	\end{equation*}
	Applying Proposition~\ref{prop:MH} to the Meixner family with
	shifted parameter $\gamma+k$ yields
	\begin{equation*}
		\theta_{n-k}(z-k)\,
		\MeixnerS{\gamma+k}{n-k}(z-k)
		\longrightarrow
		\varphi_M^{(\gamma+k)}(z-k).
	\end{equation*}
	Using \eqref{eq:phishiftk}, we have
	\begin{equation*}
		(\mu-1)^k\varphi_M^{(\gamma+k)}(z-k)
		=
		(-1)^k\varphi_M(z-k).
	\end{equation*}
	Therefore,
	\begin{equation*}
		\frac{\theta_n(z)}{\ff{n}{k}}\,
		\nab^k\Meixner n(z)
		\longrightarrow
		(-1)^k\varphi_M(z-k),
	\end{equation*}
	locally uniformly in $\C$.
	
	For the perturbation term,
	\begin{align*}
		\frac{B_n\theta_n(z)}{\ff{n}{k}}\,
		\nab^{k+1}\Meixner n(z)
		&=
		(n-k)B_n\,(\mu-1)^{k+1}
		\theta_{n-k-1}(z-k-1)
		\MeixnerS{\gamma+k+1}{n-k-1}(z-k-1).
	\end{align*}
	By Proposition~\ref{prop:MH}, the normalised Meixner factor converges
	locally uniformly to $\varphi_M^{(\gamma+k+1)}(z-k-1)$, and is therefore
	locally uniformly bounded; the scalar prefactor $(n-k)B_n\le nB_n\to0$ by
	\eqref{eq:nBn}. Hence the perturbation term also converges locally uniformly
	to zero, and consequently
	\begin{equation*}
		\frac{\theta_n(z)}{\ff{n}{k}}\,
		\nab^k\MeixnerA n(z)
		\longrightarrow
		(-1)^k\varphi_M(z-k)
		=
		(-1)^k
		\frac{1}
		{(1-\mu)^{\gamma+z-k}\Gamma(k-z)}.
	\end{equation*}
	This proves \eqref{eq:mhnablakM}.
\end{proof}

\section{A first-order functional equation for the common Mehler--Heine limit}
\label{sec:functional}

We close the asymptotic analysis with a complementary observation about the
common Mehler--Heine limit. It concerns the limiting function itself and
follows directly from its explicit reciprocal-Gamma representation; no
asymptotic reduction of the second-order difference equation satisfied by the
finite-degree modified polynomials is required, and the result does not depend
on the Uvarov structure, since the limit coincides with the classical one. Its
modest added value over the explicit representation is organisational: it
exhibits the family-dependent one-step shift satisfied by the limit, shows that
the forward-difference order $k$ enters only through the global factor $(-1)^k$
(which cancels from the equation), and identifies, among all entire solutions of
that shift equation, the specific one selected by the Mehler--Heine
normalisation. Whether the same first-order functional equation can also be
recovered as a genuine asymptotic reduction of the second-order difference
equation, after a suitable normalisation and analysis of its coefficients,
remains an interesting open question.

\begin{theorem}[First-order functional equation for the limit]
	\label{thm:reduced}
	Let
	\begin{equation*}
		\Phi_k(z)=(-1)^k\varphi_C(z)
		\quad\text{or}\quad
		\Phi_k(z)=(-1)^k\varphi_M(z),
	\end{equation*}
	be the common Mehler--Heine limit of the $k$th-order forward differences
	in Theorem~\ref{thm:mh-diff}. Then $\Phi_k$ satisfies
	\begin{equation}\label{eq:reducedC}
		z\,\Phi_k(z-1)+\Phi_k(z)=0,
		\quad\text{(Charlier)},
	\end{equation}
	or, respectively,
	\begin{equation}\label{eq:reducedM}
		z\,\Phi_k(z-1)+(1-\mu)\Phi_k(z)=0,
		\quad\text{(Meixner)},
	\end{equation}
	for all $z\in\C$.
	
	Moreover, the entire solutions of \eqref{eq:reducedC} are precisely
	\begin{equation*}
		F(z)=\frac{p(z)}{\Gamma(-z)},
	\end{equation*}
	whereas the entire solutions of \eqref{eq:reducedM} are precisely
	\begin{equation*}
		F(z)=
		\frac{p(z)(1-\mu)^{-z}}{\Gamma(-z)},
	\end{equation*}
	where, in both cases, $p$ is an arbitrary one-periodic entire function. Here,
	and throughout this section, the power $(1-\mu)^{z}$ is understood as the
	entire function
	\begin{equation*}
		(1-\mu)^{z}=\exp\!\bigl(z\,\log(1-\mu)\bigr),
		\quad 0<\mu<1,
	\end{equation*}
	with $\log(1-\mu)$ the real logarithm of the positive number $1-\mu$; this
	fixes a single-valued branch and removes any ambiguity in the factor
	$(1-\mu)^{-z}$ above.
\end{theorem}

\begin{proof}
	For the Charlier family,
	\begin{equation*}
		\Phi_k(z)
		=
		(-1)^k\frac{e^\mu}{\Gamma(-z)}.
	\end{equation*}
	The Gamma functional equation
	\begin{equation*}
		\Gamma(1-z)
		=
		-z\,\Gamma(-z),
	\end{equation*}
	is equivalently written as
	\begin{equation*}
		\frac{z}{\Gamma(1-z)}
		=
		-\frac{1}{\Gamma(-z)}.
	\end{equation*}
	Therefore,
	\begin{equation*}
		z\,\Phi_k(z-1)
		=
		(-1)^k e^\mu
		\frac{z}{\Gamma(1-z)}
		=
		-\Phi_k(z),
	\end{equation*}
	which proves \eqref{eq:reducedC} for all $z\in\C$.
	
	For the Meixner family,
	\begin{equation*}
		\Phi_k(z)
		=
		(-1)^k
		\frac{1}
		{(1-\mu)^{\gamma+z}\Gamma(-z)}.
	\end{equation*}
	Hence,
	\begin{align*}
		z\,\Phi_k(z-1)
		&=
		(-1)^k
		\frac{z}
		{(1-\mu)^{\gamma+z-1}\Gamma(1-z)}
		\\
		&=
		-(1-\mu)\Phi_k(z),
	\end{align*}
	which proves \eqref{eq:reducedM}.
	
	We now characterise the entire solution spaces. Let $F$ be an entire
	solution of \eqref{eq:reducedC} and define
	\begin{equation*}
		p(z)=\Gamma(-z)F(z),
		\quad z\notin\N_0.
	\end{equation*}
	Using
	\begin{equation*}
		F(z)=-zF(z-1),
	\end{equation*}
	together with the Gamma functional equation gives
	\begin{align*}
		p(z)
		&=
		\Gamma(-z)F(z)
		\\
		&=
		-z\,\Gamma(-z)F(z-1)
		\\
		&=
		\Gamma(1-z)F(z-1)
		=
		p(z-1).
	\end{align*}
	Thus $p$ is one-periodic on $\C\setminus\N_0$, where it is initially defined.

	It remains to examine the points of $\N_0$, where $\Gamma(-z)$ has poles and
	$p$ is a priori undefined. Setting $z=0$ in \eqref{eq:reducedC} gives
	$F(0)=0$, and evaluating the relation $F(z)=-zF(z-1)$ successively at
	$z=1,2,\ldots$ yields
	\begin{equation*}
		F(n)=0,
		\quad n\in\N_0.
	\end{equation*}
	The poles of $\Gamma(-z)$ at the nonnegative integers are all simple; since
	$F$ has a zero at each such point, the product $p(z)=\Gamma(-z)F(z)$ has only
	removable singularities on $\N_0$. Removing them defines $p$ as an entire
	function, and the periodicity $p(z)=p(z-1)$, already established on the dense
	open set $\C\setminus\N_0$, extends to all of $\C$ by analytic continuation.
	Hence $p$ is an entire one-periodic function. Conversely, if $p$ is any
	one-periodic entire function, then
	\begin{equation*}
		F(z)=\frac{p(z)}{\Gamma(-z)},
	\end{equation*}
	is entire and satisfies \eqref{eq:reducedC}.
	
	For the Meixner equation, let $F$ be an entire solution of
	\eqref{eq:reducedM} and define
	\begin{equation*}
		p(z)
		=
		(1-\mu)^z\Gamma(-z)F(z),
		\quad z\notin\N_0.
	\end{equation*}
	Since
	\begin{equation*}
		F(z)
		=
		-\frac{z}{1-\mu}F(z-1),
	\end{equation*}
	the same computation gives
	\begin{equation*}
		p(z)=p(z-1).
	\end{equation*}
	on $\C\setminus\N_0$. Moreover, \eqref{eq:reducedM} implies successively that
	$F(n)=0$ for every $n\in\N_0$; as before, the simple poles of $\Gamma(-z)$ are
	cancelled by these zeros, so $p$ has removable singularities there, extends to
	an entire function, and its periodicity extends from $\C\setminus\N_0$ to all
	of $\C$ by analytic continuation. Hence $p$ is entire and one-periodic.
	Conversely,
	every one-periodic entire function $p$ yields an entire solution
	\begin{equation*}
		F(z)
		=
		\frac{p(z)(1-\mu)^{-z}}{\Gamma(-z)}.
	\end{equation*}
	This completes the proof.
\end{proof}

Equations \eqref{eq:reducedC}--\eqref{eq:reducedM} distinguish the Charlier
and Meixner limiting functions. The coefficient of the shifted term is the
same, whereas the coefficient of $\Phi_k(z)$ is $1$ for Charlier and $1-\mu$
for Meixner.
The order $k$ enters the forward-difference limit only through the global
factor $(-1)^k$, which cancels from the functional equation. Consequently,
all fixed-order forward-difference limits, classical or mass-modified,
satisfy the same family-dependent first-order shift equation.

The one-periodic factor in the general entire solution reflects the
freedom inherent in a first-order shift equation. In the present
Mehler--Heine setting, this freedom is fixed by the particular
reciprocal-Gamma functions selected by the asymptotic normalisation,
namely
\begin{equation*}
	\frac{e^\mu}{\Gamma(-z)}
	\quad\text{and}\quad
	\frac{1}
	{(1-\mu)^{\gamma+z}\Gamma(-z)},
\end{equation*}
for the Charlier and Meixner families, respectively.

\section{Graphical illustration of the Mehler--Heine formulas}
\label{sec:GraphicalMHF}

We illustrate the higher-order Mehler--Heine formulas established in
Theorems~\ref{thm:mh-diff} and~\ref{thm:mh-nablak} for the endpoint
Uvarov-modified Charlier and Meixner families. The purpose of the numerical
experiments is to visualize how the appropriately scaled finite-degree functions
approach their limiting profiles as $n$ increases. We consider $k=2$ and $k=3$
as representative higher-order cases, although the analytical results hold for
every fixed $k\in\Nzero$.

\subsection{Numerical setup and reproducibility}
\label{subsec:numsetup}

All quantities are computed from the closed-form data of the paper, and the
computation is fully reproducible from the parameters given below. The source
code used to generate all graphical and numerical illustrations presented in
this section is publicly available at
\url{https://github.com/asorial/mehler-heine-finite-differences-code}.
We record here the discretization, the evaluation scheme for the polynomials,
the zero computation, and the arithmetic safeguards, so that every figure can
be regenerated exactly.

The monic classical polynomials $\Charlier n$ and $\Meixner n$ are evaluated
through their terminating hypergeometric representations, accumulated by the
stable $O(n)$ forward recurrence on the series coefficients rather than by
forming monomial coefficients, which avoids the ill-conditioning of the power
basis. The modified polynomials are then obtained from the exact rank-one
formula \eqref{eq:perturb}, $P_n^{A}=P_n+c_n\nab P_n$ with $c_n=D_n$ (Charlier)
or $c_n=B_n$ (Meixner) taken from their closed forms \eqref{eq:Dn}--\eqref{eq:Bn}.
Forward and backward finite differences of order $k$ are formed from the exact
binomial combinations
\begin{equation*}
	\dif^{k}P(z)=\sum_{i=0}^{k}(-1)^{k-i}\binom{k}{i}P(z+i),
	\qquad
	\nab^{k}P(z)=\sum_{i=0}^{k}(-1)^{i}\binom{k}{i}P(z-i),
\end{equation*}
and the reciprocal-Gamma normalisation $\theta_n$ of \eqref{eq:theta} and the
limit functions of \eqref{eq:phis} are evaluated with the complex Gamma function.

For the complex-plane portraits of Section~\ref{subsec:portraits} we use the
rectangular windows $z=x+iy$ with $x\in[-4,6]$ (Charlier) and $x\in[-3.5,5]$
(Meixner), and $y\in[-1.3,1.3]$ in both cases, discretized on a uniform
$1000\times300$ grid ($N_x=1000$ points in the real direction, $N_y=300$ in the
imaginary direction). On this grid we compute $\log_{10}|F(z)|$; non-finite
entries are mapped to the clip bounds ($|F|\to0$, giving $-\infty$, to the lower
bound, overflow to the upper bound), and the field is truncated to the interval
$[-3,3]$ before colouring, which also fixes the range of the superimposed level
curves $\{-1,0,1\}$. The degrees shown are $n=10,25,50$, together with the
limiting panel. At these degrees IEEE double precision is amply sufficient; we
verified this by recomputing the scaled quantities in $30$-digit multiprecision
arithmetic (\texttt{mpmath}) over the same windows, obtaining agreement to
relative error below $10^{-13}$ wherever $|F|>10^{-12}$, and agreement of the
polynomial cores to about $10^{-15}$ against $50$-digit references. This confirms
the absence of overflow, underflow, or catastrophic cancellation in the
portraits, as expected: the reciprocal-Gamma normalisation keeps every scaled
quantity of order one.

Two kinds of zeros are marked. The \emph{predicted limiting zeros} are the
integer lattice points of the corresponding limit ($\Nzero$ for the forward
panels, $k+\Nzero$ for the backward panels) that fall inside the window; they are
drawn as open circles, with crosses at the excluded points $0,1,\dots,k-1$ in the
backward case. The \emph{computed finite-degree zeros} are the real zeros of the
underlying (normalisation-free) finite-degree difference, located by scanning the
real segment of the window on a fine grid for sign changes and refining each
bracket by Brent's method to a tolerance of $10^{-13}$, followed by
de-duplication of roots closer than $10^{-6}$; they are drawn as filled dots.
Only zeros lying inside the displayed window are plotted. Working with the
normalisation-free difference for the root search avoids any spurious influence of
the poles or zeros of $\theta_n$ on the sign pattern.

For the real-axis experiment of Section~\ref{subsec:realaxis} the degrees reach
$n=250$ and the mass is the large value $A=100$. At these degrees a direct
double-precision evaluation of $\theta_n(z)=\kappa^{n}/\Gamma(n-z)$ would
overflow (through $\Gamma(n-z)$) or lose accuracy through cancellation between
the polynomial and its backward difference. We therefore carry out the entire
real-axis computation in multiprecision arithmetic (\texttt{mpmath}, $30$ working
digits): the polynomials are evaluated from their hypergeometric forms, the
reciprocal Gamma factor is computed as $\kappa^{n}\,\mathrm{rgamma}(n-z)$ using
the reciprocal-Gamma routine (so that no large $\Gamma$ value is ever formed),
and the coefficients $D_n$, $B_n$ are evaluated from their closed forms with the
kernel sums accumulated in the same precision. The scaled backward differences
$\theta_n(z)\nab^{k}P^{A}_{n}(z)/\ff{n}{k}$ then remain of order one throughout,
and no overflow, underflow, or catastrophic cancellation occurs. As a further
check, we verified at representative degrees that increasing the precision to
$50$ digits leaves the plotted curves unchanged to plotting accuracy.

\subsection{Complex-plane portraits}
\label{subsec:portraits}

For a complex-valued function $F$, with $z=x+iy$, each portrait represents $\log_{10}|F(z)|$ over the indicated rectangular region of the complex plane. The colour scale
is centred at $\log_{10}|F(z)|=0$ and clipped to the interval $[-3,3]$ in
order to provide a common visual scale for all panels. Blue tones correspond
to regions where the modulus is small, white corresponds approximately to
$|F(z)|=1$, and red tones indicate regions where the modulus is large.
Superimposed black curves represent the level sets
\begin{equation*}
	\log_{10}|F(z)|=-1,\quad
	\log_{10}|F(z)|=0,\quad
	\log_{10}|F(z)|=1,
\end{equation*}
with the level $-1$ drawn as a dashed curve and the levels $0$ and $1$ as
solid curves. These level curves make it possible to compare not only the
magnitude of the functions but also the geometry of their complex-plane
profiles.

In each figure, the first three panels correspond to $n=10$, $n=25$, and
$n=50$, while the fourth panel represents the limiting function provided by
the corresponding Mehler--Heine theorem. The panels should therefore be read
from left to right. Convergence is reflected graphically by the progressive
agreement of the colour fields, low-modulus regions, and level curves of the
finite-degree functions with those of the limiting panel. We use $\mu=0.4$
and $A=1.5$ for both families and $\gamma=2.3$ for the Meixner family.

For the forward differences, the quantities represented are
\begin{equation*}
	\theta_n(z)\,\dif^k\CharlierA n(z)
	\quad\text{and}\quad
	\theta_n(z)\,\dif^k\MeixnerA n(z).
\end{equation*}
According to Theorem~\ref{thm:mh-diff},
\begin{equation*}
	\theta_n(z)\,\dif^k\CharlierA n(z)
	\longrightarrow
	(-1)^k\varphi_C(z),
	\quad
	\theta_n(z)\,\dif^k\MeixnerA n(z)
	\longrightarrow
	(-1)^k\varphi_M(z),
\end{equation*}
locally uniformly in $\C$.

Figures~\ref{fig:portraits-forward-k2} and
\ref{fig:portraits-forward-k3} show this approximation for $k=2$ and $k=3$,
respectively. For the smaller value $n=10$, visible differences remain between
the finite-degree portrait and the limiting panel. At $n=25$ the global colour
distribution and the level curves are already closer to the limiting
geometry, and for $n=50$ the agreement becomes more pronounced. This
progressive stabilization of the graphical representation provides a direct
visual illustration of the convergence established in
Theorem~\ref{thm:mh-diff}.

\begin{figure}[htbp]
	\centering
	\begin{subfigure}{\textwidth}
		\centering
		\includegraphics[width=\linewidth]{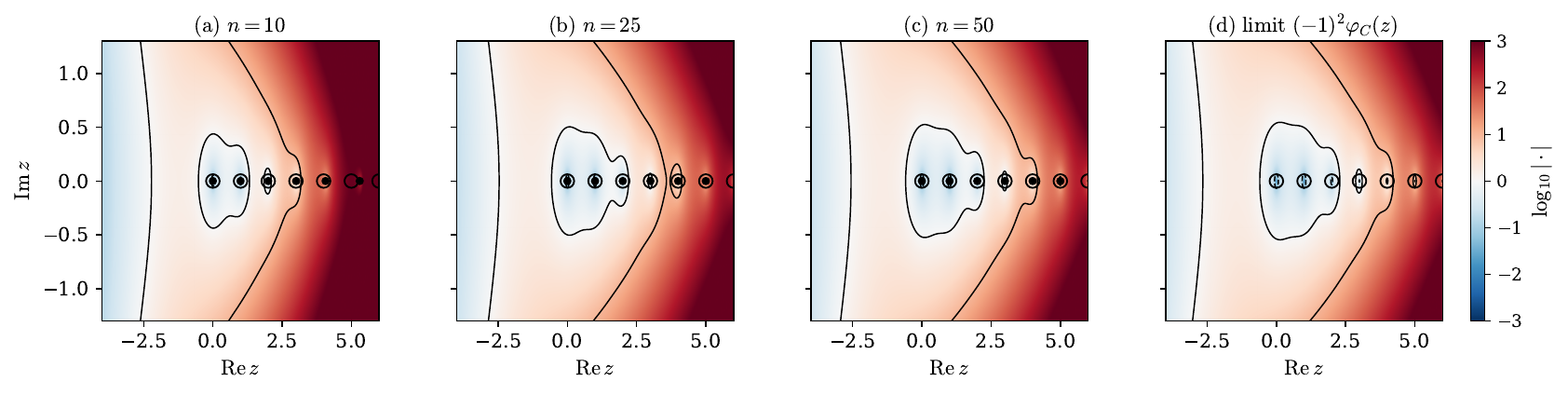}
		\caption{Charlier:
			$\log_{10}\!\big|\theta_n(z)\,\dif^{k}\CharlierA n(z)\big|$,
			with limiting profile $(-1)^{k}\varphi_C(z)$.}
		\label{fig:portraits-forward-k2-charlier}
	\end{subfigure}\\[1.2ex]
	\begin{subfigure}{\textwidth}
		\centering
		\includegraphics[width=\linewidth]{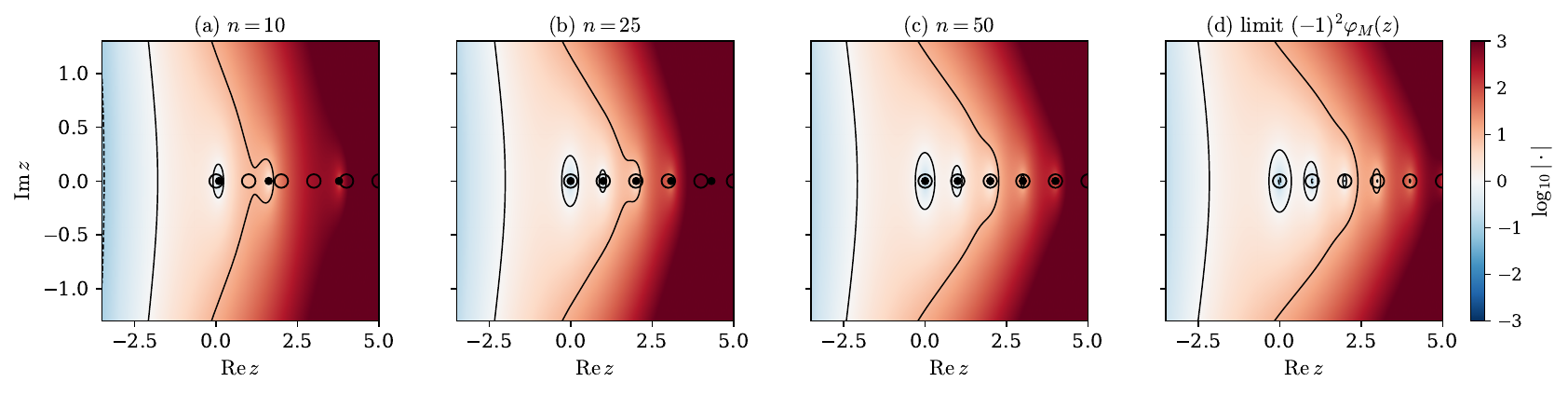}
		\caption{Meixner:
			$\log_{10}\!\big|\theta_n(z)\,\dif^{k}\MeixnerA n(z)\big|$,
			with limiting profile $(-1)^{k}\varphi_M(z)$.}
		\label{fig:portraits-forward-k2-meixner}
	\end{subfigure}
	\caption{Complex-plane log-modulus portraits of the scaled forward
		differences for $k=2$, with $\mu=0.4$, $A=1.5$, and $\gamma=2.3$ in
		the Meixner case. From left to right, the panels correspond to
		$n=10,25,50$ and to the limiting function from
		Theorem~\ref{thm:mh-diff}. The colour represents the clipped value of
		$\log_{10}|F(z)|$ on $[-3,3]$, and the black curves are the level sets
		$\log_{10}|F(z)|=-1,0,1$. As $n$ increases, both the colour field and
		the level-set geometry progressively approach those of the limiting
		panel.}
	\label{fig:portraits-forward-k2}
\end{figure}

\begin{figure}[htbp]
	\centering
	\begin{subfigure}{\textwidth}
		\centering
		\includegraphics[width=\linewidth]{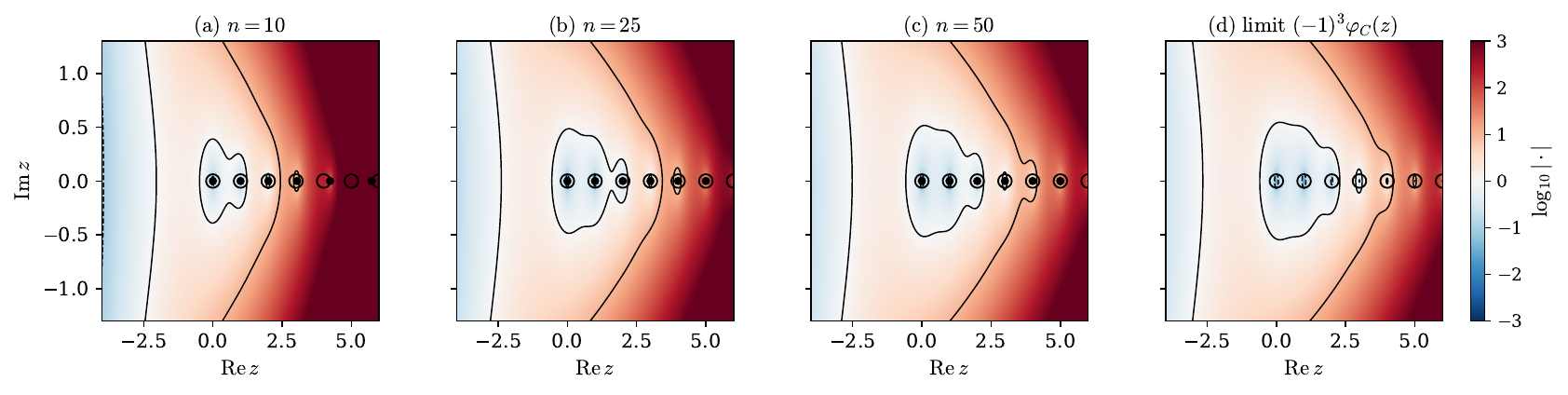}
		\caption{Charlier:
			$\log_{10}\!\big|\theta_n(z)\,\dif^{k}\CharlierA n(z)\big|$,
			with limiting profile $(-1)^{k}\varphi_C(z)$.}
		\label{fig:portraits-forward-k3-charlier}
	\end{subfigure}\\[1.2ex]
	\begin{subfigure}{\textwidth}
		\centering
		\includegraphics[width=\linewidth]{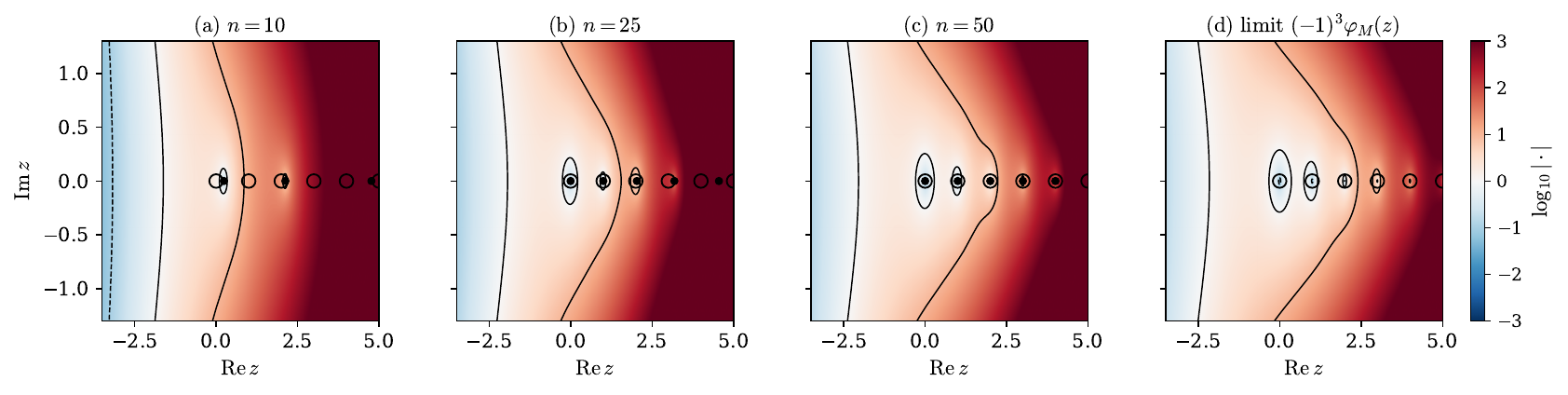}
		\caption{Meixner:
			$\log_{10}\!\big|\theta_n(z)\,\dif^{k}\MeixnerA n(z)\big|$,
			with limiting profile $(-1)^{k}\varphi_M(z)$.}
		\label{fig:portraits-forward-k3-meixner}
	\end{subfigure}
	\caption{Complex-plane log-modulus portraits of the scaled forward
		differences for $k=3$. The parameters are $\mu=0.4$, $A=1.5$, and
		$\gamma=2.3$ in the Meixner case. The panels correspond to
		$n=10,25,50$ and to the limiting function from
		Theorem~\ref{thm:mh-diff}. The progressive agreement of the colour
		distribution and the level curves with the fourth panel illustrates the
		approach to the Mehler--Heine limit as $n$ increases.}
	\label{fig:portraits-forward-k3}
\end{figure}

For the backward differences, the quantities represented are
\begin{equation*}
	\frac{\theta_n(z)}{\ff{n}{k}}\,\nab^k\CharlierA n(z)
	\quad\text{and}\quad
	\frac{\theta_n(z)}{\ff{n}{k}}\,\nab^k\MeixnerA n(z).
\end{equation*}
Theorem~\ref{thm:mh-nablak} gives
\begin{equation*}
	\frac{\theta_n(z)}{\ff{n}{k}}\,
	\nab^k\CharlierA n(z)
	\longrightarrow
	(-1)^k\varphi_C(z-k),
	\quad
	\frac{\theta_n(z)}{\ff{n}{k}}\,
	\nab^k\MeixnerA n(z)
	\longrightarrow
	(-1)^k\varphi_M(z-k),
\end{equation*}
locally uniformly in $\C$. In contrast with the forward case, the limiting
profile is translated by $k$ units in the argument. This translation is
visible in the complex-plane geometry itself, since the principal structures
of the colour field and the level curves are displaced accordingly.

Figures~\ref{fig:portraits-backward-k2} and
\ref{fig:portraits-backward-k3} display the corresponding approximation for
$k=2$ and $k=3$. As in the forward case, the finite-degree portraits become
progressively closer to the limiting representation. The differences visible
for $n=10$ decrease at $n=25$, while the $n=50$ portrait reproduces more
closely the spatial distribution of the modulus and the characteristic level
curves of the limiting panel. The figures therefore make both the convergence
and the translation $z\mapsto z-k$ visually apparent.

\begin{figure}[htbp]
	\centering
	\begin{subfigure}{\textwidth}
		\centering
		\includegraphics[width=\linewidth]{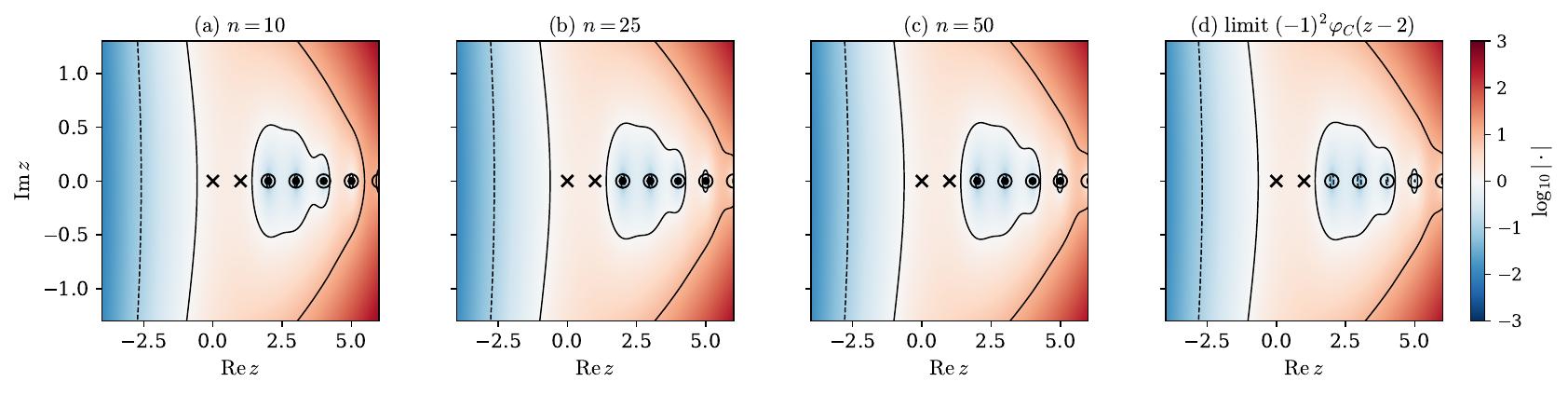}
		\caption{Charlier:
			$\log_{10}\!\left|\dfrac{\theta_n(z)}{\ff{n}{k}}\,
			\nab^{k}\CharlierA n(z)\right|$,
			with limiting profile $(-1)^{k}\varphi_C(z-k)$.}
		\label{fig:portraits-backward-k2-charlier}
	\end{subfigure}\\[1.2ex]
	\begin{subfigure}{\textwidth}
		\centering
		\includegraphics[width=\linewidth]{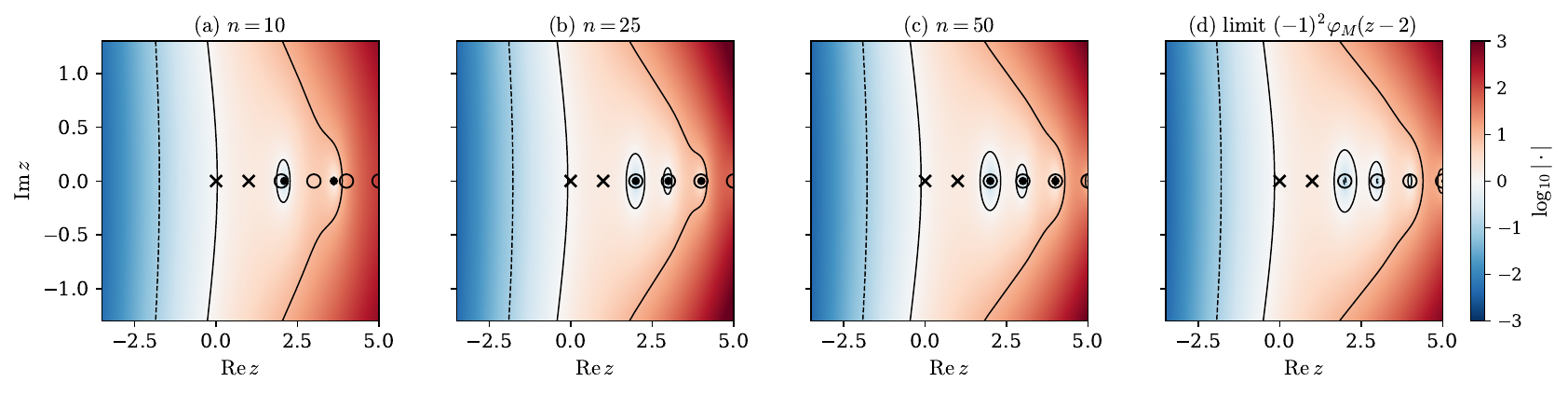}
		\caption{Meixner:
			$\log_{10}\!\left|\dfrac{\theta_n(z)}{\ff{n}{k}}\,
			\nab^{k}\MeixnerA n(z)\right|$,
			with limiting profile $(-1)^{k}\varphi_M(z-k)$.}
		\label{fig:portraits-backward-k2-meixner}
	\end{subfigure}
	\caption{Complex-plane log-modulus portraits of the scaled backward
		differences for $k=2$, with $\mu=0.4$, $A=1.5$, and $\gamma=2.3$ in
		the Meixner case. From left to right, the panels correspond to
		$n=10,25,50$ and to the limiting function from
		Theorem~\ref{thm:mh-nablak}. As the degree increases, the finite-degree
		colour fields and level curves progressively reproduce the translated
		limiting profile displayed in the fourth panel.}
	\label{fig:portraits-backward-k2}
\end{figure}

\begin{figure}[htbp]
	\centering
	\begin{subfigure}{\textwidth}
		\centering
		\includegraphics[width=\linewidth]{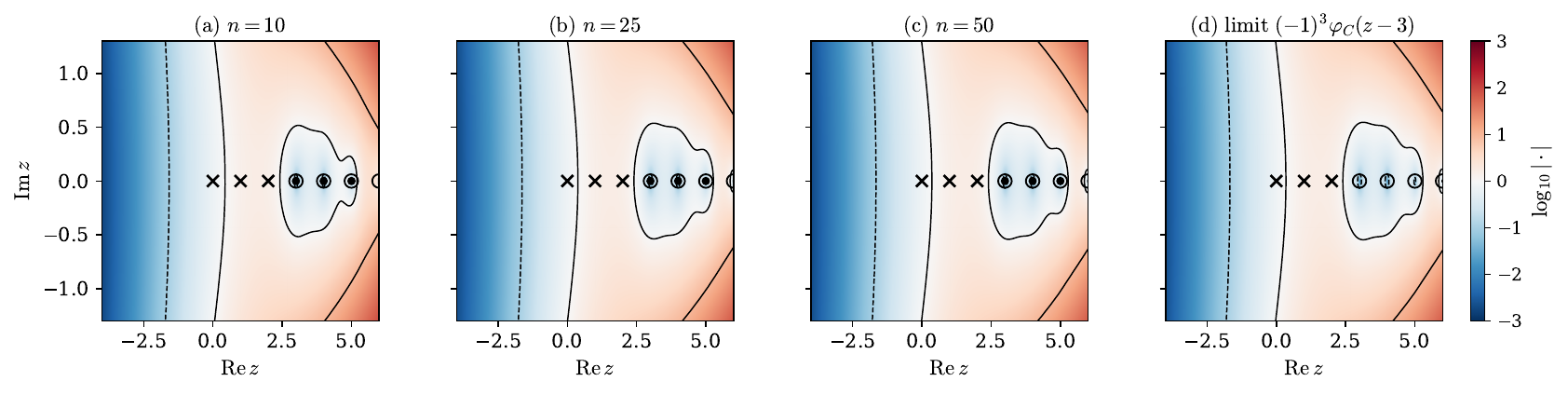}
		\caption{Charlier:
			$\log_{10}\!\left|\dfrac{\theta_n(z)}{\ff{n}{k}}\,
			\nab^{k}\CharlierA n(z)\right|$,
			with limiting profile $(-1)^{k}\varphi_C(z-k)$.}
		\label{fig:portraits-backward-k3-charlier}
	\end{subfigure}\\[1.2ex]
	\begin{subfigure}{\textwidth}
		\centering
		\includegraphics[width=\linewidth]{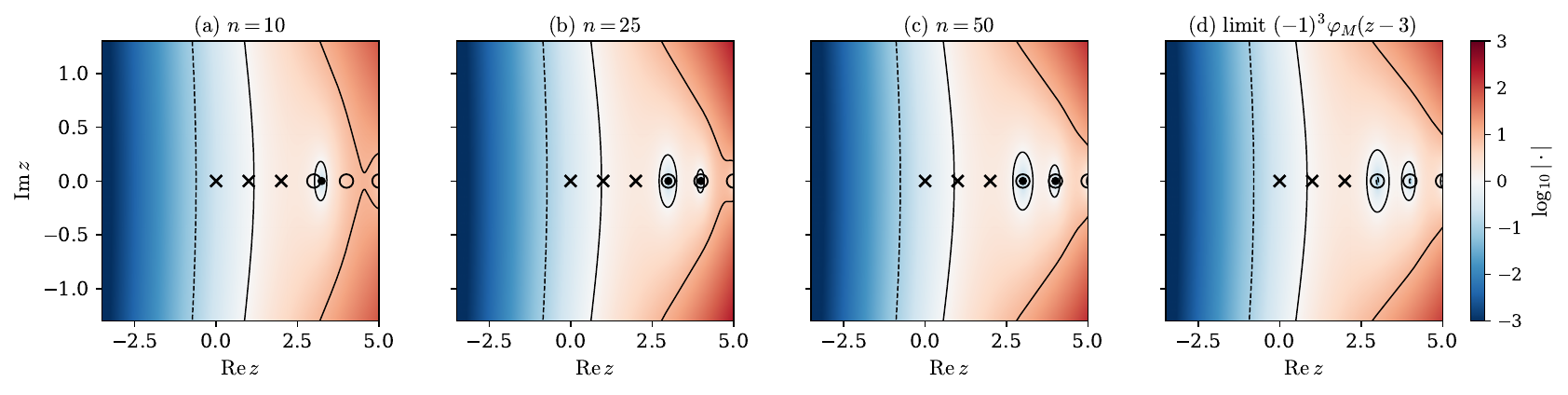}
		\caption{Meixner:
			$\log_{10}\!\left|\dfrac{\theta_n(z)}{\ff{n}{k}}\,
			\nab^{k}\MeixnerA n(z)\right|$,
			with limiting profile $(-1)^{k}\varphi_M(z-k)$.}
		\label{fig:portraits-backward-k3-meixner}
	\end{subfigure}
	\caption{Complex-plane log-modulus portraits of the scaled backward
		differences for $k=3$. The parameters are $\mu=0.4$, $A=1.5$, and
		$\gamma=2.3$ in the Meixner case. The panels correspond to
		$n=10,25,50$ and to the limiting function from
		Theorem~\ref{thm:mh-nablak}. The progressive similarity between the
		finite-degree panels and the fourth panel makes the convergence toward
		the shifted Mehler--Heine profile directly visible.}
	\label{fig:portraits-backward-k3}
\end{figure}

\subsection{Convergence along the real axis}
\label{subsec:realaxis}

The complex-plane portraits are complemented by a direct comparison along the
real axis. Figure~\ref{fig:realaxis} represents the scaled backward differences
from Theorem~\ref{thm:mh-nablak} for $k=3$ and
$n=50,100,150,200,250$. Each coloured curve corresponds to a different
finite degree, whereas the thick black curve represents the corresponding
Mehler--Heine limiting function. Thus, unlike the complex-plane portraits,
where convergence is observed through colour fields and level sets, these
plots allow the vertical separation between the finite-degree functions and
their limit to be seen directly.

For the Charlier family, the curves are compared with
$(-1)^k\varphi_C(z-k)$, whereas for the Meixner family they are compared with
$(-1)^k\varphi_M(z-k)$. We deliberately choose the comparatively large fixed
mass $A=100$ in order to make the finite-degree perturbation clearly visible.
The displayed real intervals are $z\in[-4,7]$ for Charlier (with $\mu=3.5$) and
$z\in[-1.5,6]$ for Meixner (with $\gamma=5$, $\mu=0.4$).
For the smaller values of $n$, the corresponding curves remain visibly
separated from the black limiting profile. As $n$ increases, this separation
decreases and the finite-degree curves follow more closely the shape,
oscillatory behaviour, and principal extrema of the limiting function. The
sequence $n=50,100,150,200,250$ therefore provides a direct one-dimensional
visualization of the same convergence observed in the complex-plane
portraits.

\begin{figure}[htbp]
	\centering
	\begin{subfigure}{0.49\textwidth}
		\centering
		\includegraphics[width=\linewidth]{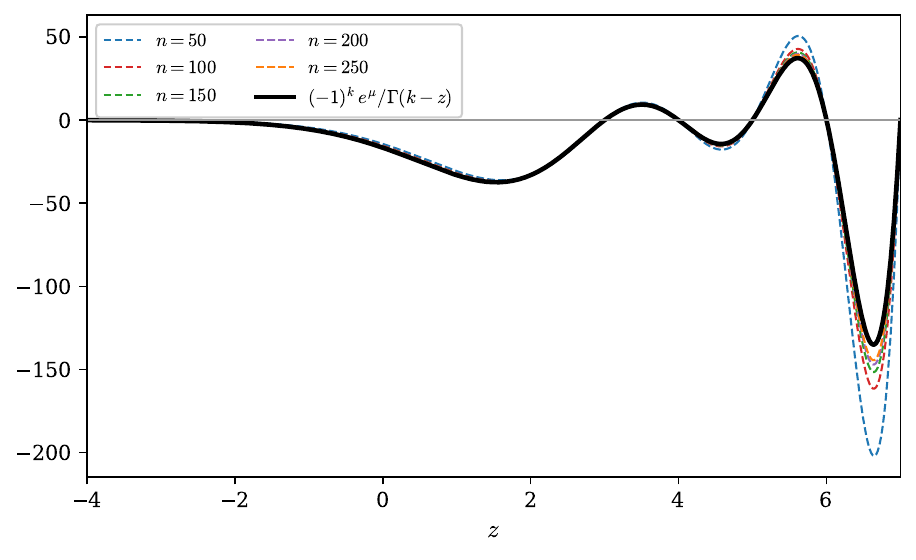}
		\caption{Charlier, $\mu=3.5$, $A=100$; limit
			$(-1)^{k}e^{\mu}/\Gamma(k-z)$.}
		\label{fig:realaxis-charlier}
	\end{subfigure}\hfill
	\begin{subfigure}{0.49\textwidth}
		\centering
		\includegraphics[width=\linewidth]{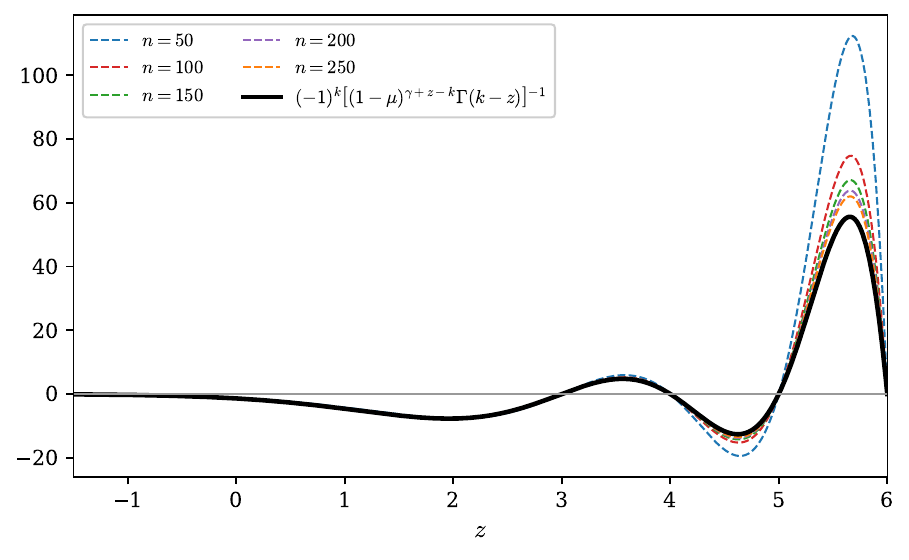}
		\caption{Meixner, $\gamma=5$, $\mu=0.4$, $A=100$; limit
			$(-1)^{k}\big[(1-\mu)^{\gamma+z-k}\Gamma(k-z)\big]^{-1}$.}
		\label{fig:realaxis-meixner}
	\end{subfigure}
	\caption{Real-axis comparison of the scaled backward differences
		$\theta_n(z)\nab^{k}P^{A}_{n}(z)/\ff{n}{k}$ for $k=3$ and
		$n=50,100,150,200,250$ with the limiting functions from
		Theorem~\ref{thm:mh-nablak}. The coloured curves represent the
		finite-degree functions and the thick black curve the corresponding
		limit. As $n$ increases, the finite-degree curves progressively approach
		the limiting profile throughout the displayed interval.}
	\label{fig:realaxis}
\end{figure}

Taken together, the complex-plane and real-axis representations provide a
direct numerical visualization of Theorems~\ref{thm:mh-diff} and
\ref{thm:mh-nablak}. For both the Charlier and Meixner families, increasing
$n$ produces graphical representations that become progressively closer to
the corresponding Mehler--Heine limiting representation. In the complex
plane, this is reflected in the stabilization of the colour fields and level
curves, while along the real axis it appears as the decreasing visual
separation between the finite-degree curves and the limiting profile.
Forward differences preserve the reciprocal-Gamma geometry up to the factor
$(-1)^k$, whereas backward differences additionally exhibit the translation
$z\mapsto z-k$. In both cases, the limiting representation is independent of
the fixed endpoint mass $A$.

\section{Conclusions}
\label{sec:conclu}

We have determined the effect of a fixed endpoint Uvarov mass on the
higher-order finite-difference Mehler--Heine asymptotics of the monic Charlier
and Meixner families. The key mechanism is the rank-one connection formula
\eqref{eq:perturb}. The corresponding perturbation coefficients $D_n$ and
$B_n$ exhibit factorial decay in the Charlier case and exponential decay with
an algebraic prefactor in the Meixner case, as described in
\eqref{eq:Dnasy} and \eqref{eq:Bnasy}. In both families this decay dominates
the algebraic factors generated by any fixed number of finite differences.
Consequently, for every fixed $A\geq0$ and $k\in\Nzero$, the mass-modified
and classical families have the same locally uniform Mehler--Heine limits in
$\C$. The result concerns fixed $A$ and fixed $k$; no uniformity with respect
to regimes such as $A=A_n$ or $k=k_n$ is asserted.

The direction of differencing nevertheless remains visible in the limiting
structure. Equations~\eqref{eq:mhdiffC}--\eqref{eq:mhdiffM} show that forward
differences preserve the reciprocal-Gamma profile, up to the factor
$(-1)^k$. By contrast, the backward formulas
\eqref{eq:mhnablakC}--\eqref{eq:mhnablakM} require the additional
normalisation by $\ff{n}{k}$ and translate the limiting argument according to
$z\mapsto z-k$. Thus, although the fixed endpoint mass disappears from the
limit, the choice of the finite-difference operator continues to determine
the geometry of the resulting Mehler--Heine profile.

The first-order shift equations \eqref{eq:reducedC} and
\eqref{eq:reducedM} provide a complementary description of the forward
limits. Their entire solutions are reciprocal-Gamma factors multiplied by
one-periodic entire functions, while the Mehler--Heine normalisation selects
the constant periodic factors corresponding to the Charlier and Meixner
families. This characterization isolates a functional structure shared by the
classical and mass-modified systems and clarifies why the fixed forward
difference order appears in the limit only through the factor $(-1)^k$.

The endpoint character of the perturbation is also relevant. Point-mass
modifications at $x=0$ may substantially alter the finite-degree Charlier and
Meixner families \cite{alvareznodarse1995}, whereas exterior Sobolev
perturbations can retain the perturbation point explicitly in the
Mehler--Heine limit \cite{SoriaMichel2026}. For the pure endpoint Uvarov
modification considered here, however, the fixed mass parameter leaves no
trace in any of the fixed-order limiting formulas. This comparison indicates
that the local asymptotic behaviour depends not only on the presence of a
perturbation, but also on its type and location.

The numerical illustrations provide a direct visualization of the analytical
results. In the complex-plane portraits, the representations corresponding to
$n=10$, $n=25$, and $n=50$ become progressively closer to the limiting
representation predicted by the corresponding Mehler--Heine formula. This
approach is visible in the stabilization of the colour fields and of the
associated level-set geometry for both the Charlier and Meixner families and
for both forward and backward differences. The real-axis plots provide the
same comparison for $n=50,100,150,200,250$: as the degree increases, the
finite-degree curves increasingly reproduce the shape of the limiting
function. The use of the comparatively large fixed mass $A=100$ in these
plots further illustrates that the finite-degree effect of the perturbation
may remain visible while its contribution disappears from the asymptotic
limit.

Taken together, the analytical and graphical results reveal a clear
scale-separation mechanism. The endpoint mass modifies the finite-degree
polynomials, but its contribution is suppressed at the Mehler--Heine scale by
the rapid decay of the perturbation coefficients. At the same time, the
finite-difference operator leaves a persistent signature: forward differences
retain the original reciprocal-Gamma geometry, whereas backward differences
produce the translated profile associated with $z\mapsto z-k$.

Several natural extensions remain open. Quantitative remainder estimates
would strengthen the locally uniform convergence by providing explicit error
bounds. Regimes with $A=A_n$ or $k=k_n$ could identify thresholds beyond which
the scale separation established here no longer holds. Multiple mass points,
mixed endpoint and exterior perturbations, and analogous questions for other
discrete or $q$-orthogonal families also provide natural directions for
further study.

%%%%%%%%%%%%%%%%%%%%%%%%%%%%%%%%%%%%%%%%%%%%%%%%%%%%%%%%%%%%%%%%%%%%%%%%%%%%%%%%%%%%%%%%%%%%%%%%%%%%%%%

%%%%%%%%%%%%%%%%%%%%%%%%%%%%%%%%%%%%%%%%%%%%%%%%%%%%%%%%%%%%%%%%%%%%%%%%%%%%%%%%%%%%%%%%%%%%%%%%%%%%%%%

%%%%%%%%%%%%%%%%%%%%%%%%%%%%%%%%%%%%%%%%%%%%%%%%%%%%%%%%%%%%%%%%%%%%%%%%%%%%%%%%%%%%%%%%%%%%%%%%%%%%%%%
%%%%%%%%%%%%%%%%%%%%%%%%%%%%%%%%%%%%%%%%%%%%%%%%%%%%%%%%%%%%%%%%%%%%%%%%%%%%%%%%%%%%%%%%%%%%%%%%%%%%%%%

%%%%%%%%%%%%%%%%%%%%%%%%%%%%%%%%%%%%%%%%%%%%%%%%%%%%%%%%%%%%%%%%%%%%%%%%%%%%%%%%%%%%%%%%%%%%%%%%%%%%%%%

\section*{Acknowledgements}

The authors would like to thank the Department of Quantitative Methods at Universidad Loyola Andalusia for providing an excellent research environment and institutional support during the development of this work.

\section*{Declarations}

\subsection*{Ethical Approval}
Not applicable.

\subsection*{Consent to Participate}
Not applicable.

\subsection*{Consent to Publish}
Not applicable.

\subsection*{Data Availability Statement}
No external datasets were used in this study. The source code used to generate
all the figures in Section~\ref{sec:GraphicalMHF}, together with the corresponding
parameter values and supplementary numerical tables, is available in the
Supplementary Material accompanying this article.

\subsection*{Author Contributions}
Conceptualization, A.S.-L.; methodology, A.S.-L. and J.-M.; software, A.S.-L. and J.-M.; validation, A.S.-L. and J.-M.; formal analysis, A.S.-L. and J.-M.; investigation, A.S.-L. and J.-M.; resources, A.S.-L.; data curation, A.S.-L. and J.-M.; writing--original draft preparation, A.S.-L.; writing--review and editing, A.S.-L. and J.-M.; visualization, A.S.-L.; supervision, A.S.-L.; project administration, A.S.-L. All authors have read and agreed to the published version of the manuscript.

\subsection*{Funding}
This research received no external funding.

\subsection*{Competing Interests}
The authors declare that they have no competing interests.

%%%%%%%%%%%%%%%%%%%%%%%%%%%%%%%%%%%%%%%%%%%%%%%%%%%%%%%%%%%%%%%%%%%%%%%%%%%%%%%%%%%%%%%%%%%%%%%%%%%%%%%

%%%%%%%%%%%%%%%%%%%%%%%%%%%%%%%%%%%%%%%%%%%%%%%%%%%%%%%%%%%%%%%%%%%%%%%%%%%%%%%%%%%%%%%%%%%%%%%%%%%%%%%


\begin{thebibliography}{99}

\bibitem{almarero}
M.~Alfaro, F.~Marcell\'an, M.~L.~Rezola, and A.~Ronveaux.
\newblock On orthogonal polynomials of Sobolev type: algebraic properties and zeros.
\newblock {\em SIAM Journal on Mathematical Analysis}, 23(3):737--757, 1992.
\newblock \href{https://doi.org/10.1137/0523038}{\path{doi:10.1137/0523038}}.

\bibitem{alvareznodarse1995}
R.~\'Alvarez-Nodarse, A.~G.~Garc\'ia, and F.~Marcell\'an.
\newblock On the properties for modifications of classical orthogonal polynomials of discrete variables.
\newblock {\em Journal of Computational and Applied Mathematics}, 65:3--18, 1995.
\newblock \href{https://doi.org/10.1016/0377-0427(95)00097-6}{\path{doi:10.1016/0377-0427(95)00097-6}}.

\bibitem{agm}
I.~\'Area, E.~Godoy, and F.~Marcell\'an.
\newblock Inner products involving differences: the Meixner--Sobolev polynomials.
\newblock {\em Journal of Difference Equations and Applications}, 6(1):1--31, 2000.
\newblock \href{https://doi.org/10.1080/10236190008808211}{\path{doi:10.1080/10236190008808211}}.

\bibitem{agmmb}
I.~\'Area, E.~Godoy, F.~Marcell\'an, and J.~J.~Moreno-Balc\'azar.
\newblock Ratio and Plancherel--Rotach asymptotics for Meixner--Sobolev orthogonal polynomials.
\newblock {\em Journal of Computational and Applied Mathematics}, 116:63--75, 2000.
\newblock \href{https://doi.org/10.1016/S0377-0427(99)00281-2}{\path{doi:10.1016/S0377-0427(99)00281-2}}.

\bibitem{B1995}
H.~Bavinck.
\newblock On polynomials orthogonal with respect to an inner product involving differences.
\newblock {\em Journal of Computational and Applied Mathematics}, 57:17--27, 1995.
\newblock \href{https://doi.org/10.1016/0377-0427(93)E0231-A}{\path{doi:10.1016/0377-0427(93)E0231-A}}.

\bibitem{B1995gen}
H.~Bavinck.
\newblock On polynomials orthogonal with respect to an inner product involving differences (the general case).
\newblock {\em Applicable Analysis}, 59:233--240, 1995.
\newblock \href{https://doi.org/10.1080/00036819508840402}{\path{doi:10.1080/00036819508840402}}.

\bibitem{B1996}
H.~Bavinck.
\newblock A difference operator of infinite order with the Sobolev-type Charlier polynomials as eigenfunctions.
\newblock {\em Indagationes Mathematicae (N.S.)}, 7(3):281--291, 1996.
\newblock \href{https://doi.org/10.1016/0019-3577(96)83721-9}{\path{doi:10.1016/0019-3577(96)83721-9}}.

\bibitem{Costas-2022}
R.~S.~Costas-Santos, A.~Soria-Lorente, and J.-M.~Vilaire.
\newblock On polynomials orthogonal with respect to an inner product involving higher-order differences: the Meixner case.
\newblock {\em Mathematics}, 10(11):1952, 2022.
\newblock \href{https://doi.org/10.3390/math10111952}{\path{doi:10.3390/math10111952}}.

\bibitem{Dominici2015}
D.~Dominici.
\newblock Mehler--Heine type formulas for Charlier and Meixner polynomials.
\newblock {\em Ramanujan Journal}, 39:271--289, 2016.
\newblock \href{https://doi.org/10.1007/s11139-014-9665-5}{\path{doi:10.1007/s11139-014-9665-5}}.

\bibitem{Dominici2018}
D.~Dominici.
\newblock Mehler--Heine type formulas for Charlier and Meixner polynomials II. Higher order terms.
\newblock {\em Journal of Classical Analysis}, 12(1):9--13, 2018.
\newblock \href{https://doi.org/10.7153/jca-2018-12-02}{\path{doi:10.7153/jca-2018-12-02}}.

\bibitem{domo}
D.~Dominici and J.~J.~Moreno-Balc\'azar.
\newblock Asymptotic analysis of a family of Sobolev orthogonal polynomials related to the generalized Charlier polynomials.
\newblock {\em Journal of Approximation Theory}, 293:105918, 2023.
\newblock \href{https://doi.org/10.1016/j.jat.2023.105918}{\path{doi:10.1016/j.jat.2023.105918}}.

\bibitem{gamamo}
G.~Filipuk, J.~F.~Ma\~nas-Ma\~nas, J.~J.~Moreno-Balc\'azar, and C.~Rodr\'iguez-Perales.
\newblock Second--order difference equation for quasi--orthogonal polynomials related to Hahn difference operator.
\newblock {\em Journal of Mathematical Analysis and Applications}, 557(1):130268, 2026.
\newblock \href{https://doi.org/10.1016/j.jmaa.2025.130268}{\path{doi:10.1016/j.jmaa.2025.130268}}.

\bibitem{HS2019}
E.~J.~Huertas and A.~Soria-Lorente.
\newblock New analytic properties of nonstandard Sobolev-type Charlier orthogonal polynomials.
\newblock {\em Numerical Algorithms}, 82:41--68, 2019.
\newblock \href{https://doi.org/10.1007/s11075-018-0593-0}{\path{doi:10.1007/s11075-018-0593-0}}.

\bibitem{Ismail05}
M.~E.~H.~Ismail.
\newblock {\em Classical and Quantum Orthogonal Polynomials in One Variable}.
\newblock Encyclopedia of Mathematics and its Applications, Vol.~98.
\newblock Cambridge University Press, Cambridge, 2005.
\newblock \href{https://doi.org/10.1017/CBO9781107325982}{\path{doi:10.1017/CBO9781107325982}}.

\bibitem{KLS2010}
R.~Koekoek, P.~A.~Lesky, and R.~F.~Swarttouw.
\newblock {\em Hypergeometric Orthogonal Polynomials and Their $q$-Analogues}.
\newblock Springer Monographs in Mathematics.
\newblock Springer, Berlin, 2010.
\newblock \href{https://doi.org/10.1007/978-3-642-05014-5}{\path{doi:10.1007/978-3-642-05014-5}}.

\bibitem{maxu}
F.~Marcell\'an and Y.~Xu.
\newblock On Sobolev orthogonal polynomials.
\newblock {\em Expositiones Mathematicae}, 33:308--352, 2015.
\newblock \href{https://doi.org/10.1016/j.exmath.2014.10.002}{\path{doi:10.1016/j.exmath.2014.10.002}}.

\bibitem{MR1990}
F.~Marcell\'an and A.~Ronveaux.
\newblock On a class of polynomials orthogonal with respect to a discrete Sobolev inner product.
\newblock {\em Indagationes Mathematicae (N.S.)}, 1(4):451--464, 1990.
\newblock \href{https://doi.org/10.1016/0019-3577(90)90013-D}{\path{doi:10.1016/0019-3577(90)90013-D}}.

\bibitem{mbDMS}
J.~J.~Moreno-Balc\'azar.
\newblock $\Delta$-Meixner--Sobolev orthogonal polynomials: Mehler--Heine type formula and zeros.
\newblock {\em Journal of Computational and Applied Mathematics}, 284:228--234, 2015.
\newblock \href{https://doi.org/10.1016/j.cam.2014.11.018}{\path{doi:10.1016/j.cam.2014.11.018}}.

\bibitem{mbtpmp}
J.~J.~Moreno-Balc\'azar, T.~E.~P\'erez, and M.~A.~Pi\~nar.
\newblock A generating function for nonstandard orthogonal polynomials involving differences: the Meixner case.
\newblock {\em Ramanujan Journal}, 25:21--35, 2011.
\newblock \href{https://doi.org/10.1007/s11139-010-9254-1}{\path{doi:10.1007/s11139-010-9254-1}}.

\bibitem{nikiforov1991classical}
A.~F.~Nikiforov, V.~B.~Uvarov  and S.~K.~Suslov. 
\newblock {\em Classical Orthogonal Polynomials of a Discrete Variable}.
\newblock Springer Series in Computational Physics.
\newblock Springer, Berlin, 1991.
\newblock \href{https://doi.org/10.1007/978-3-642-74748-9}{\path{doi:10.1007/978-3-642-74748-9}}.

\bibitem{Rebocho2022}
M.~N.~Rebocho.
\newblock On the second-order holonomic equation for Sobolev-type orthogonal polynomials.
\newblock {\em Applicable Analysis}, 101(1):314--336, 2022.
\newblock \href{https://doi.org/10.1080/00036811.2020.1742881}{\path{doi:10.1080/00036811.2020.1742881}}.

\bibitem{khol}
R.~Koekoek.
\newblock {\em Generalizations of the Classical Laguerre Polynomials and Some $q$-Analogues}.
\newblock Ph.D. thesis, Delft University of Technology, 1990.
\newblock \href{https://resolver.tudelft.nl/uuid:5545d761-01b1-4f0f-8e55-3956da27d878}
{\path{https://resolver.tudelft.nl/uuid:5545d761-01b1-4f0f-8e55-3956da27d878}}.

\bibitem{mapepi}
F.~Marcell\'an, M.~Alfaro, and M.~L.~Rezola.
\newblock Orthogonal polynomials on Sobolev spaces: old and new directions.
\newblock {\em Journal of Computational and Applied Mathematics}, 48(1--2):113--131, 1993.
\newblock \href{https://doi.org/10.1016/0377-0427(93)90318-6}{\path{doi:10.1016/0377-0427(93)90318-6}}.

\bibitem{ANbook}
R.~\'Alvarez-Nodarse.
\newblock {\em Polinomios hipergeom\'etricos cl\'asicos y $q$-polinomios},
volume~26 of {\em Monograf\'ias del Seminario Matem\'atico Garc\'ia de Galdeano}.
\newblock Prensas Universitarias de Zaragoza, Zaragoza, 2003.

\bibitem{SoriaMichel2026}
A.~Soria-Lorente and J.~Michel.
\newblock On generating functions and Mehler--Heine formulas for discrete Charlier and Meixner Sobolev-type orthogonal polynomials.
\newblock arXiv:2607.12889v1 [math.CA], 2026.
\newblock \href{https://arxiv.org/abs/2607.12889}{\path{arXiv:2607.12889}}.

\end{thebibliography}
\end{document}